\documentclass[a4paper, 12pt]{article}

\usepackage[english]{babel}
\usepackage{authblk}

\usepackage[a4paper,top=30mm,left=25mm,right=25mm,marginparwidth=1.75cm]{geometry}

\usepackage{empheq}
\usepackage{graphicx}
\usepackage[colorinlistoftodos]{todonotes}
\usepackage{amsmath}
\usepackage{mathtools}
\usepackage{upgreek}
\usepackage{amssymb}
\usepackage{amsfonts}
\usepackage{amsthm}
\usepackage{enumitem}
\usepackage{mathrsfs}
\usepackage{fontenc}
\usepackage{verbatim}
\usepackage{booktabs}
\usepackage{subcaption}
\usepackage{framed}
\usepackage{algorithm}
\usepackage{bbm}
\usepackage{algpseudocode}
\usepackage[title, titletoc, header]{appendix}
\usepackage{color}
\usepackage{multirow}
\usepackage{xcolor}
\usepackage{microtype}
\usepackage[colorlinks=true, citecolor=blue, linkcolor=blue, urlcolor=blue, hidelinks]{hyperref}
\usepackage{cleveref}
\usepackage{caption}
\usepackage{mfirstuc}
\usepackage{microtype}
\usepackage[figuresright]{rotating}

\newcommand{\N}{\mathbb N}
\newcommand{\E}{\mathbb E}
\newcommand{\F}{\mathbb F}
\newcommand{\R}{\mathbb R}
\newcommand{\bP}{\mathbb P}

\newcommand{\cF}{\mathcal F}
\newcommand{\cL}{\mathcal L}
\newcommand{\cS}{\mathcal S}
\newcommand{\cU}{\mathcal U}
\newcommand{\cV}{\mathcal V}
\newcommand{\cX}{\mathcal X}
\newcommand{\cY}{\mathcal Y}
\newcommand{\cZ}{\mathcal Z}

\numberwithin{equation}{section}
\theoremstyle{plain}
\newtheorem{thm}{Theorem}[section]

\newtheorem{prop}{Proposition}[section]

\newtheorem{cor}{Corollary}[section]

\newtheorem{assu}{Assumption}

\newtheorem{defn}{Definition}[section]
\newtheorem*{rem}{Remark}

\theoremstyle{definition}

\theoremstyle{remark}

\title{A Posteriori Error Analysis for Decoupled Neural Approximations of Fully Coupled FBSDEs with Control Mismatch}

\author[1]{Xichuan \textsc{Zhang}
}
\affil[1]{Intelligent Game and Decision Lab, Beijing 100091, China}

\date{}

\begin{document}

\maketitle


\begin{abstract}
This paper develops an a posteriori error analysis framework for decoupled neural approximations of fully coupled forward--backward stochastic differential equations (FBSDEs).
It provides an a posteriori error-analysis for the idealized discrete adapted trajectory.
The main feature of the proposed formulation is the use of an auxiliary control process in the forward coefficients,
which may differ from the backward component approximated by the neural network.
This decoupling is useful in practical deep learning implementations,
but it creates a control mismatch that must be included in the error analysis.

We first establish a continuous-time stability estimate for fully coupled FBSDEs under perturbations of the drift, diffusion, generator, terminal condition, and auxiliary control input.
We then transfer this estimate to the discrete-time setting and derive computable a posteriori error bounds depending only on the terminal defect, the pathwise residual, and the control mismatch.
When the auxiliary control is identified with the backward approximation, the mismatch term vanishes and the bound reduces to the standard two-term form.
Numerical experiments on a linear--quadratic FBSDE with an explicit reference solution and a multidimensional Burgers-type FBSDE without a reference solution illustrate the diagnostic role of the proposed indicators and the contribution of the mismatch penalty to the consistency and reproducibility of the numerical approximations

\textbf{Keywords} fully coupled FBSDE, deep BSDE method, a posteriori error estimate, control mismatch, deep neural networks

\textbf{MSC codes} 65C30, 60H35, 68T07, 65M15
\end{abstract}

\section{Introduction}
\label{sec-Introduction}

Backward stochastic differential equations (BSDEs) were systematically introduced by Pardoux and Peng~\cite{PardouxPeng1990},
who proved the existence and uniqueness of nonlinear BSDEs under Lipschitz conditions.
For fully coupled forward--backward stochastic differential equations (FBSDEs),
in which the forward drift and diffusion may depend on the backward component $Y$ and the control process $Z$,
fundamental well-posedness results were subsequently established by Peng and Wu~\cite{hu1995solution,PengWu1999} and Ma and Yong~\cite{MaYong1999}.
These equations provide a probabilistic representation for a wide range of problems in stochastic control,
mathematical finance, recursive utility,
and quasilinear parabolic partial differential equations (PDEs) through nonlinear Feynman--Kac type formulas~\cite{ElKarouiQuenez1997,finance1997,peng1990stochasticmax,DuffieEpstein1992a,pengPDE}.

Despite their theoretical importance,
the numerical solution of FBSDEs remains challenging,
especially in high dimensions.
Classical grid-based methods,
such as finite difference and finite element methods~\cite{ma1994solving,BouchardTouzi2004,BouchardEkeland2004,BenderZhang2008,Zhao2016Multistep,Bender2012Monte,GobetLabart2008},
suffer from the curse of dimensionality~\cite{Bellman1957}.
Regression-based Monte Carlo methods and multistep schemes have greatly improved the numerical treatment of BSDEs and FBSDEs,
but their implementation and accuracy may still deteriorate rapidly as the dimension increases~\cite{GobetLemorWarin2005,GobetTurkedjiev2016}.

A major development in this direction is the deep BSDE method proposed by Han, Jentzen and E~\cite{HanJentzenE2018}.
By reformulating the solution of an FBSDE as a stochastic optimization problem and parameterizing the unknown control process with neural networks,
the deep BSDE method provides an effective approach for high-dimensional nonlinear PDEs and stochastic systems.
Since then, a growing literature has developed deep learning-based methods for BSDEs, FBSDEs, and related PDEs~\cite{RaissiPerdikarisKarniadakis2019,WeinanHanJentzen2017,pham2018deep_1,pham2018deep_2,Peng_FBSDE_numerical,GermainMikaelWarin2022}.
On the theoretical side, Han and Long~\cite{HanLong2020} established a posteriori error estimates and convergence results for deep BSDE methods under suitable approximation assumptions.
Related developments include extensions to FBSDEs with jumps, optimal stopping problems, multi-FBSDE formulations, and mean-field or McKean--Vlasov type systems~\cite{GnoattoOberprillerPicarelli2025,GaoEtAl2023,AnderssonAnderssonOosterlee2025,ReisingerStockingerZhang2024}.

The closest work to the present paper is the a posteriori error analysis of fully coupled McKean--Vlasov FBSDEs by Reisinger, Stockinger and Zhang~\cite{ReisingerStockingerZhang2024}.
Their framework treats a more general class of law-dependent fully coupled systems and establishes residual-type a posteriori estimates for time-discretized McKean--Vlasov FBSDEs.
Since a classical fully coupled FBSDE can be regarded as a special case without law dependence,
their results are closely related to the setting considered here;
however, they do not address the control mismatch that arises from decoupled neural parametrizations.

Previous work~\cite{Peng_FBSDE_numerical} proposed three deep learning algorithms for high-dimensional fully coupled FBSDEs.
The second algorithm therein introduced an auxiliary control process $U$ and its training loss already included a penalty on the mismatch between $U$ and $Y$.
The loss therefore consisted only of the terminal defect and the control-mismatch penalty.
While that two-term formulation proved effective in practice,
no rigorous a posteriori error analysis was available,
and the connection between the training objective and a provable error bound remained unclear.

In the present paper,
we build upon the decoupling idea and develop a comprehensive a posteriori error estimation framework that explicitly incorporates the control mismatch.
Crucially,
we augment the algorithm with an explicit pathwise residual term $\cL_R$ that penalizes deviations from the backward dynamic equation at each time step.
Together with the terminal defect $\cL_T$ and the control mismatch $\cL_U$, this yields a three‑component loss that fully mirrors the structure of the a posteriori error bound we derive.
The main contributions are as follows.

First, we establish a continuous-time stability estimate for fully coupled FBSDEs under perturbations in the drift, diffusion, generator, terminal condition, and auxiliary control input.
The estimate explicitly contains the mismatch term measuring the distance between the backward component and the auxiliary control.

Second, we transfer this stability estimate to the discrete-time setting and obtain a computable a posteriori error bound for arbitrary adapted neural approximations.
The bound depends only on the terminal defect, the pathwise residual, and the control mismatch,
all of which can be evaluated from the numerical output.
Therefore, the estimate is genuinely a posteriori in the sense that it does not require knowledge of the true solution.

Third, we show through numerical experiments that the mismatch penalty is not merely a theoretical artifact.
The first example is a high-dimensional linear--quadratic FBSDE for which an exact solution can be obtained through the associated Riccati equation.
This allows us to compare the computable indicators with the true approximation error.
The second example is a multidimensional Burgers-type FBSDE for which no closed-form reference solution is used.
In this case, the a posteriori indicators serve as diagnostic tools for comparing different training strategies and for assessing the effect of removing individual loss components.

The remainder of the paper is organized as follows.
Section~\ref{sec-preliminaries} introduces the mathematical setting of fully coupled FBSDEs and the neural network approximation framework.
Section~\ref{sec-error-analysis} presents the main continuous stability estimate and the discrete-time a posteriori error bounds.
Section~\ref{sec-proofs} contains the proofs of the main theorems.
Section~\ref{sec-numerics} reports the numerical experiments and discusses the behavior of the proposed indicators.
Section~\ref{sec-conclusion} concludes the paper and outlines possible directions for future work.

\section{Preliminaries}
\label{sec-preliminaries}

This section introduces the basic concepts needed for this article, including the forward-backward stochastic differential equation (FBSDEs), deep neural networks and universal approximation.
This preparatory discussion aims to establish the necessary theoretical foundation and computational framework for our subsequent algorithmic development.

\subsection{Mathematical Framework of Fully-Coupled FBSDEs}
\label{ssec-FBSDES}

Let $\R^n$ denote the $n$-dimensional Euclidean space equipped with the standard inner product $\langle x,y \rangle$ and the Euclidean norm $\|x\|^2=\langle x,x \rangle$ for all $x,y\in\R^n$.
Let $\R^{m \times n}$ be the Hilbert space of all $(m \times n)$-matrices endowed with the inner product
\[
    \langle A, B \rangle = \mathrm{tr}(AB^\top), \qquad \forall A, B \in \R^{m \times n},
\]
and the induced norm $\|A\|^2 = \langle A, A \rangle$.

Let $T>0$ be fixed,
and let $(\Omega, \cF, \F, \bP)$ be a complete probability space equipped with a $d$-dimensional standard Brownian motion $\{B_t\}_{0 \le t \le T}$.
We denote by $\F=\{\cF_t\}_{0\le t\le T}$ the natural filtration generated by $\{B_t\}$, augmented by all $\bP$-null sets in $\cF$.

Let $L^2(\cF_t;\R^n)$ be the space of $\cF_t$-measurable $\R^n$-valued square-integrable random variables.
We define the space $M^2(0,T;\R^n)$ as the set of all $\F$-progressively measurable processes $v:[0,T]\times\Omega\to\R^n$ satisfying $\E\bigl[\int_0^T\|v(t)\|^2\,dt\bigr] < \infty.$
This space becomes a Hilbert space when endowed with the inner product
$\langle u(\cdot), v(\cdot) \rangle_{M^2} \coloneqq \E\bigl[\int_0^T\langle u(t), v(t) \rangle_{\R^n}\,dt\bigr].$

A fully coupled forward–backward stochastic differential equation defined on $ (\Omega, \cF, \F, \bP) $ has the following general form
\begin{equation}\label{eq-FBSDE}
    \begin{cases}
        dX_t = b(t,X_t,Y_t,Z_t)\,dt + \sigma(t,X_t,Y_t,Z_t)\,dB_t, \\
        -dY_t = f(t,X_t,Y_t,Z_t)\,dt - Z_t\,dB_t, \\
        X_0=a, \qquad Y_T=g(X_T),
    \end{cases}
\end{equation}
where $X(\cdot)\in M^2(0,T;\R^n)$, $Y(\cdot)\in M^2(0,T;\R^m)$,
and $Z(\cdot)\in M^2(0,T;\R^{m \times d})$ are $\F$-adapted stochastic processes,
$a$ and $g(X_T)$ are the initial and terminal conditions, respectively.
$b$ and $\sigma$ denote the forward SDE drift coefficient and diffusion coefficient,
and $f$ denotes the generator of the backward SDE.
When $b(\cdot)$ and $\sigma(\cdot)$ do not depend on the processes $Y(\cdot),Z(\cdot)$,
the above fully coupled FBSDE can be simplified as follows
\begin{equation}\label{eq-BSDE}
    \begin{cases}
        dX_t = b(t,X_t)\,dt + \sigma(t,X_t)\,dB_t, \\
        -dY_t = f(t,X_t,Y_t,Z_t)\,dt - Z_t\,dB_t, \\
        X_0=a, \qquad Y_T=g(X_T),
    \end{cases}
\end{equation}
which is a basic uncoupled BSDE.

\begin{defn}
\label{def-solution-FBSDE}
    A triple of processes $(X(\cdot),Y(\cdot),Z(\cdot))\in M^2(0,T;\R^n\times\R^m\times \R^{m\times d}) $ is called an adapted solution of \eqref{eq-FBSDE} if it satisfies the \eqref{eq-FBSDE} $\bP$-almost surely on $[0,T]$.
\end{defn}

Let
\begin{equation*}
    u =
    \begin{pmatrix}
        x \\ y \\ z
    \end{pmatrix} \in \R^n\times\R^m\times\R^{m\times d}, \qquad A(t,u) =
    \begin{pmatrix}
        -G^\top f \\ Gb \\ G\sigma
    \end{pmatrix}(t,u),
\end{equation*}
where $G$ is given a full rank matrix satisfied $G\sigma = (G\sigma_1,\cdots,G\sigma_d)$ and
$\left\langle u^1, u^2 \right\rangle = \left\langle x^1, x^2 \right\rangle + \left\langle y^1, y^2 \right\rangle + \left\langle z^1, z^2 \right\rangle.$

\begin{assu}
\label{assu-1}
\begin{enumerate}
    \item $A(t,u)$ is uniformly Lipschitz with respect to $u$,
    $A(\cdot,u)$ is in $M^2(0,T)$ for each $u$.
    $g(x)$ on $x\in\R^n$ is uniformly Lipschitz with respect to $x$,
    and is in $L^2(\cF_T;\R^n )$.
    The Lipschitz constant denotes as $K>0$.
    \item $\phi=b,\sigma,f$ is uniformly Lipschitz with respect to $u$
    \item $\forall\,u,\bar{u}\in\R^n\times\R^m\times\R^{m\times d}$ and $x,\bar{x}$,
    there exist constants $\beta_1,\beta_2,\mu\ge 0$ with $\beta_1+\beta_2>0$ and $\mu+\beta_2>0$,
    such that:
    \begin{align*}
        \langle A(t,u)-A(t,\bar{u}), u-\bar{u} \rangle &\le -\beta_1|G(x-\bar{x})|^2-\beta_2\bigl( |G^\top(y-\bar{y})|^2+|G^\top(z-\bar{z})|^2 \bigr), \\
        \langle g(x)-g(\bar{x}), G(x-\bar{x}) \rangle &\ge \mu|G(x-\bar{x})|^2,
    \end{align*}
    where $\beta_1>0$ and $\mu>0$ if $m>n$, while $\beta_2>0$ if $n>m$.
    \end{enumerate}
\end{assu}

\begin{prop}
\label{prop-existence}
    Under the \cref{assu-1},
    the above FBSDE \eqref{eq-FBSDE} admits a unique process solution $(X(\cdot),Y(\cdot),Z(\cdot))$ in $M^2(0,T;\R^n\times\R^m\times\R^{m\times d})$.
\end{prop}

\begin{proof}[Proof Sketch]
    The existence follows from a Picard iteration scheme,
    while uniqueness is guaranteed by the Lipschitz continuity and monotonicity conditions in \cref{assu-1}.
    The proof of the theorem is given in \cite{hu1995solution,wu1999Fully}.
\end{proof}

A foundational connection between stochastic processes and deterministic partial differential equations is established via the \textit{nonlinear Feynman-Kac formula}.
This formula provides a probabilistic representation for solutions to a broad class of parabolic PDEs through the lens of (forward-)backward stochastic differential equations.

For the more general \textit{fully coupled} FBSDE system as defined in~\eqref{eq-FBSDE},
where the drift and diffusion of the forward process depend explicitly on $(Y_t,Z_t)$,
the associated PDE becomes \textit{quasilinear}.
The solution $u(t,x)$, if sufficiently smooth,
satisfies a quasilinear parabolic PDE whose precise form can be found in~\cite{MaYong1999} or~\cite{PengWu1999}.

\subsection{Feedforward Neural Networks and Universal Approximation}
\label{ssec-NNet}

Let $d_{\text{in}},d_{\text{out}}\in\N$,
a feedforward neural network (FNN) is a hierarchical function mapping an input $x \in \R^{d_{\text{in}}}$ to an output $f_{\theta}(x) \in \R^{d_{\text{out}}}$ through $L$ layers of parameterized transformations.
Mathematically, it is a function defined by
\begin{equation}\label{eq-FNN}
    \begin{aligned}
        h_0(x)&=x, \\
        h_{\ell}(x)&=\rho(A_{\ell}h_{\ell-1}+b_{\ell}), \qquad \text{for }\ell=1,\ldots,L-1,\\
        f_{\theta}(x)&=h_L(x)=A_{L}h_{L-1}+b_L,
    \end{aligned}
\end{equation}
where $A_{\ell} \in \R^{d_\ell \times d_{\ell-1}}$ and $b_{\ell} \in \R^{d_\ell}$ are learnable parameters,
$\rho: \R \to \R$ is a component-wise activation function.

The architecture is denoted by $\cS = (d_0, d_1, \dots, d_L)$,
where $d_0=d_{\text{in}}$ and $d_L=d_{\text{out}}$ are the input and output dimensions, respectively.
The depth of the network is denoted by $D(\cS)=L$,
the width is denoted by $\|\cS\|_{\infty}=\max_{1\leq i \leq L} d_{i} $,
and the total number of the parameters of network is denoted by $|\cS| = \sum_{i=0}^{L-1} (d_i \times d_{i+1} + d_{i+1})$.
$\theta=\{A_{\ell},b_{\ell}\}_{0\leq\ell\leq L}$ are called the weight parameters,
and $\Theta:=\{\theta\}$ is denoted as the space of parameters.
Without loss of clarity, we also use $\cU(x;\theta)$ or $\cV(x;\theta)$ to represent the neural network in this following paper.

In this paper, we employ a bounded, smooth activation function, such as the logistic function or the hyperbolic tangent,
\[
\rho(x)=\frac{1}{1+e^{-x}}, \quad \text{or} \quad \rho(x)=\frac{e^x-e^{-x}}{e^x+e^{-x}},
\]
as the activation function $\rho$ for all hidden layers. The output layer uses the linear activation.

A fundamental theoretical justification for employing neural networks to solve functional equations is the universal approximation property.

\begin{thm}[Universal Approximation Theorem]
\label{thm-univ-approx}
Let $K \subset \R^{d_{\text{in}}}$ be a compact set and let $f: K \to \R^{d_{\text{out}}}$ be a continuous function.
For any $\varepsilon > 0$, there exists a feedforward neural network $\cU(\cdot;\theta)$ with architecture $\cS = (d_{\text{in}}, d_1, d_{\text{out}})$, i.e.\ a single hidden layer containing $d_1$ neurons, and a non-constant, bounded, continuous activation function $\rho$, such that
\[
\sup_{x \in K} \bigl| f(x) - \cU(x;\theta) \bigr| < \varepsilon .
\]
The number of hidden neurons $d_1$ can be chosen finite but sufficiently large.
\end{thm}

\begin{rem}
\label{rem-approx}
In our deep learning algorithm we shall use deeper networks to parameterize the backward processes $Y_t$, $Z_t$ and the auxiliary control $U_t$.
Combined with the stability estimates of Section~\ref{sec-error-analysis},
this provides motivation that a sufficiently well-trained network can yield an accurate solution in the sense of the a posteriori error bounds.
The quantitative relationship between the network architecture and the resulting approximation error is beyond the scope of the present paper and will be investigated separately.
\end{rem}

\section{The Main Results of Error Analysis}
\label{sec-error-analysis}

In this section we establish the core error estimates of the paper.
The presentation is organized from the general to the specific.
We first prove in \S\ref{ssec-cont-general} a continuous-time stability result for fully coupled FBSDEs under the most general perturbations:
the forward drift and diffusion may depend on an auxiliary control $\cU$ that differs from the backward component $\cY$.
The classical stability bound for identical dynamics follows immediately as a corollary.
In \S\ref{ssec-disc-general} we transfer this continuous theory to the discrete-time setting,
deriving computable a posteriori error bounds for arbitrary adapted discrete approximations.
Again, the general statement allows a mismatch between the backward variable and the control surrogate used in the forward coefficients;
the standard case without mismatch is recovered as a special case.

\subsection{Continuous Stability under General Perturbations}
\label{ssec-cont-general}

Let $(\cX,\cY,\cZ)$ and $\cU$ be adapted processes satisfying the perturbed dynamics
\begin{equation}\label{eq-perturbed-general}
    \begin{cases}
        d\cX_t = \bigl[b(t,\cX_t,\cU_t,\cZ_t)+\alpha_t\bigr]\,dt
        + \bigl[\sigma(t,\cX_t,\cU_t,\cZ_t)+\beta_t\bigr]\,dB_t,\\
        -d\cY_t = \bigl[f(t,\cX_t,\cY_t,\cZ_t)+\gamma_t\bigr]\,dt
        - \cZ_t\,dB_t,\\
        \cX_0 = a,\qquad \cY_T = g(\cX_T) - \eta,
    \end{cases}
\end{equation}
where $\alpha,\beta,\gamma$ are progressively measurable perturbations,
$\eta$ is a square-integrable terminal defect,
and $\cU$ is an auxiliary process that enters both the forward drift $b$ and the forward diffusion $\sigma$.
When $\cU\equiv\cY$,
system \eqref{eq-perturbed-general} reduces to the standard perturbed FBSDE with identical dynamics;
otherwise $\cU$ models a control surrogate that may deviate from $\cY$ in both coefficients.

\begin{thm}[General continuous stability]
\label{thm-continuous-general}
Suppose Assumption~\ref{assu-1} holds.
Let $(X,Y,Z)$ be the solution of \eqref{eq-FBSDE} and $(\cX,\cY,\cZ)$ be the solution of \eqref{eq-perturbed-general}.
Then there exists a constant $C>0$,
depending only on the constants in Assumption~\ref{assu-1} and $T$,
such that
\begin{equation}\label{eq-stability-general}
    \begin{aligned}
        \sup_{0\le t\le T}&\E\bigl[|X_t-\cX_t|^2+|Y_t-\cY_t|^2\bigr]
        + \E\Bigl[\int_0^T|Z_t-\cZ_t|^2\,dt\Bigr]\\
        &\le C\Bigl(
        \E\bigl[|\eta|^2\bigr]
        + \E\Bigl[\int_0^T(|\alpha_t|^2+|\beta_t|^2+|\gamma_t|^2)\,dt\Bigr] 
        + \E\Bigl[\int_0^T|\cY_t-\cU_t|^2\,dt\Bigr]
        \Bigr).
    \end{aligned}
\end{equation}
\end{thm}

\begin{cor}[Stability under identical dynamics]
\label{cor-continuous-identical}
If $\cU\equiv\cY$ in \eqref{eq-perturbed-general}, then the control mismatch integral vanishes and estimate \eqref{eq-stability-general} reduces to
\begin{equation*}
    \begin{aligned}
        \sup_{0\le t\le T}&\E\bigl[|X_t-\cX_t|^2+|Y_t-\cY_t|^2\bigr]
        + \E\Bigl[\int_0^T|Z_t-\cZ_t|^2\,dt\Bigr]\\
        & \le C\Bigl(
        \E\bigl[|\eta|^2\bigr]
        + \E\Bigl[\int_0^T(|\alpha_t|^2+|\beta_t|^2+|\gamma_t|^2)\,dt\Bigr]
        \Bigr).
    \end{aligned}   
\end{equation*}
\end{cor}

\begin{proof}[Proof of Corollary~\ref{cor-continuous-identical}]
Setting $\cU_t = \cY_t$ for all $t$ in \eqref{eq-perturbed-general} makes the control mismatch term $\E\bigl[\int_0^T|\cY_t-\cU_t|^2\,dt\bigr]$ identically zero.
The stated bound then follows directly from Theorem~\ref{thm-continuous-general}.
\end{proof}

\begin{rem}
Theorem~\ref{thm-continuous-general} shows that the solution map of the fully coupled FBSDE is stable in the mean-square sense with respect to simultaneous perturbations in drift, diffusion, generator, terminal condition, and---when present---forward control mismatch.
The control mismatch penalty $\E\bigl[\int_0^T|\cY_t-\cU_t|^2\,dt\bigr]$ accounts for the discrepancy between the true backward component and the auxiliary control used in both $b$ and $\sigma$.
\end{rem}

\subsection{A Posteriori Error Estimates for Discrete Approximations}
\label{ssec-disc-general}

We now translate the continuous stability theory into computable error bounds for discrete-time schemes.
Let $\pi=\{0=t_0<\dots<t_N=T\}$ be a partition of $[0,T]$ with $\Delta t_i:=t_{i+1}-t_i$ and $|\pi|:=\max_i\Delta t_i$.
Given an $\cF_{t_i}$-adapted discrete trajectory $\{(\hat{X}_i,\hat{Y}_i,\hat{Z}_i,\hat{U}_i)\}_{i=0}^N$ with $\hat{X}_0=a$,
we denote by $(\bar{X},\bar{Y},\bar{Z},\bar{U})$ its piecewise constant interpolation.
Define the local residuals
\begin{align}
    R_i^X &:= \hat{X}_{i+1}-\hat{X}_i - b(t_i,\hat{X}_i,\hat{U}_i,\hat{Z}_i)\,\Delta t_i - \sigma(t_i,\hat{X}_i,\hat{U}_i,\hat{Z}_i)\,\Delta B_i,\label{eq-res-X}\\
    R_i^Y &:= \hat{Y}_{i+1}-\hat{Y}_i + f(t_i,\hat{X}_i,\hat{Y}_i,\hat{Z}_i)\,\Delta t_i - \hat{Z}_i\,\Delta B_i,\label{eq-res-Y}
\end{align}
the terminal mismatch $\eta := g(\hat{X}_N)-\hat{Y}_N$, and the computable error indicators
\begin{equation}
    \mathfrak{R}_\pi := \sum_{i=0}^{N-1}\E\Bigl[\frac{|R_i^X|^2}{\Delta t_i}+\frac{|R_i^Y|^2}{\Delta t_i}\Bigr],\qquad
    \mathfrak{P}_\pi := \sum_{i=0}^{N-1}\E\bigl[|\hat{Y}_i-\hat{U}_i|^2\bigr]\,\Delta t_i.
\end{equation}

To convert the piecewise constant approximation into a continuous-time perturbed system,
we require the following mild time-regularity of the coefficients.

\begin{assu}[$1/2$-H\"older continuity in time]\label{assu-holder}
There exists $L_H>0$ such that for all $s,t\in[0,T]$ and all $(x,y,z)$,
\[
    |b(t,x,y,z)-b(s,x,y,z)| + |\sigma(t,x,y,z)-\sigma(s,x,y,z)| + |f(t,x,y,z)-f(s,x,y,z)|
    \le L_H|t-s|^{1/2}.
\]
\end{assu}

\begin{assu}
    \label{assu-2}
The discrete approximation $\{(\hat X_i,\hat Y_i,\hat Z_i,\hat U_i)\}_{i=0}^N$ satisfies the uniform moment bound
\[
\sup_{0\le i\le N} \E\Big[|\hat X_i|^2 + |\hat Y_i|^2 + |\hat Z_i|^2 + |\hat U_i|^2\Big] \le M,
\]
where $M>0$ is a constant that may depend on the network architecture and the training outcome,
but is independent of the time step $\Delta t$.
\end{assu}

\begin{thm}[General a posteriori error estimate]
\label{thm-discrete-general}
Suppose Assumptions~\ref{assu-1}--~\ref{assu-2} hold,
and let $(X,Y,Z)$ be the solution of \eqref{eq-FBSDE}.
Then there exists $C>0$, depending only on the constants in Assumption~\ref{assu-1}, $L_H$, $M$ and $T$,
such that
\begin{equation}\label{eq-discrete-general}
    \sup_{0\le t\le T}\E\bigl[|X_t-\bar{X}_t|^2+|Y_t-\bar{Y}_t|^2\bigr]
    + \E\Bigl[\int_0^T|Z_t-\bar{Z}_t|^2\,dt\Bigr]
    \le C\bigl(\E\bigl[|\eta|^2\bigr] + \mathfrak{R}_\pi + \mathfrak{P}_\pi + |\pi|\bigr).
\end{equation}
In particular, on the grid points,
\begin{equation}\label{eq-discrete-general-grid}
    \max_{0\le i\le N}\E\bigl[|X_{t_i}-\hat{X}_i|^2+|Y_{t_i}-\hat{Y}_i|^2\bigr]
    + \E\Bigl[\int_0^T|Z_t-\bar{Z}_t|^2\,dt\Bigr]
    \le C\bigl(\E\bigl[|\eta|^2\bigr] + \mathfrak{R}_\pi + \mathfrak{P}_\pi + |\pi|\bigr).
\end{equation}
\end{thm}

\begin{cor}[A posteriori estimate without control mismatch]
\label{cor-discrete-identical}
If $\hat{U}_i\equiv\hat{Y}_i$ for all $i$,
then $\mathfrak{P}_\pi=0$ and the bound \eqref{eq-discrete-general} simplifies to
\[
    \sup_{0\le t\le T}\E\bigl[|X_t-\bar{X}_t|^2+|Y_t-\bar{Y}_t|^2\bigr]
    + \E\Bigl[\int_0^T|Z_t-\bar{Z}_t|^2\,dt\Bigr]
    \le C\bigl(\E\bigl[|\eta|^2\bigr] + \mathfrak{R}_\pi + |\pi|\bigr).
\]
\end{cor}

\begin{proof}[Proof of Corollary~\ref{cor-discrete-identical}]
If $\hat{U}_i = \hat{Y}_i$ for every $i$,
then by definition $\mathfrak{P}_\pi = \sum_{i=0}^{N-1}\E\bigl[|\hat{Y}_i-\hat{U}_i|^2\bigr]\Delta t_i = 0$.
Substituting this into \eqref{eq-discrete-general} yields the claimed estimate.
\end{proof}

\begin{cor}[Euler-type schemes]
\label{cor-euler-general}
Assume the discrete trajectory is generated by the forward Euler updates
\begin{align}
    \hat{X}_{i+1} &= \hat{X}_i + b(t_i,\hat{X}_i,\hat{U}_i,\hat{Z}_i)\,\Delta t_i + \sigma(t_i,\hat{X}_i,\hat{U}_i,\hat{Z}_i)\,\Delta B_i, \label{eq-euler-forward}\\
    \hat{Y}_{i+1} &= \hat{Y}_i - f(t_i,\hat{X}_i,\hat{Y}_i,\hat{Z}_i)\,\Delta t_i + \hat{Z}_i\,\Delta B_i + \varepsilon_i^Y, \label{eq-euler-backward}
\end{align}
where $\varepsilon_i^Y$ is an $\cF_{t_{i+1}}$-measurable local defect.
Then $\mathfrak{R}_\pi = \sum_{i=0}^{N-1}\E\bigl[|\varepsilon_i^Y|^2\bigr]/\Delta t_i$,
and \eqref{eq-discrete-general} becomes
\begin{equation*}
    \begin{aligned}
        \sup_{0\le t\le T}\E\bigl[|X_t-\bar{X}_t|^2+|Y_t-\bar{Y}_t|^2\bigr]
        & + \E\Bigl[\int_0^T|Z_t-\bar{Z}_t|^2\,dt\Bigr]\\
        & \le C\Bigl(\E\bigl[|\eta|^2\bigr] + \sum_{i=0}^{N-1}\frac{\E\bigl[|\varepsilon_i^Y|^2\bigr]}{\Delta t_i} + \mathfrak{P}_\pi + |\pi|\Bigr).
    \end{aligned}
\end{equation*}
If in addition $\varepsilon_i^Y=0$ and $\hat{U}_i=\hat{Y}_i$,
we recover the classical $O(|\pi|)$ consistency bound.
\end{cor}

\begin{proof}[Proof of Corollary~\ref{cor-euler-general}]
Under the forward Euler update \eqref{eq-euler-forward},
the forward residual satisfies $R_i^X = 0$ for all $i$.
From \eqref{eq-euler-backward}, we obtain $R_i^Y = \varepsilon_i^Y$.
Consequently,
\[
\mathfrak{R}_\pi = \sum_{i=0}^{N-1}\E\Bigl[\frac{|R_i^X|^2}{\Delta t_i} + \frac{|R_i^Y|^2}{\Delta t_i}\Bigr]
= \sum_{i=0}^{N-1}\frac{\E\bigl[|\varepsilon_i^Y|^2\bigr]}{\Delta t_i}.
\]
Inserting this expression for $\mathfrak{R}_\pi$ into Theorem~\ref{thm-discrete-general} yields the first bound.
If additionally $\varepsilon_i^Y = 0$ and $\hat{U}_i = \hat{Y}_i$, then $\mathfrak{R}_\pi = 0$ and $\mathfrak{P}_\pi = 0$, giving the $O(|\pi|)$ estimate.
\end{proof}

\begin{rem}
The indicators $\eta$, $\mathfrak{R}_\pi$, and $\mathfrak{P}_\pi$ are computable directly from the discrete output without knowledge of the true solution.
They serve as practical a posteriori criteria for model selection, adaptive refinement, and stopping in deep learning algorithms for fully coupled FBSDEs.
The mismatch penalty $\mathfrak{P}_\pi$ is particularly relevant when separate network parametrizations are used for the backward component $\hat{Y}$ and the control surrogate $\hat{U}$ that enters both $b$ and $\sigma$.
If one further sets $\mathfrak R_\pi=0$ (i.e., discards the pathwise residual penalty) while keeping $\hat{U}$ decoupled from $\hat{Y}$,
the scheme reduces to the second case of~\cite{Peng_FBSDE_numerical}.
In that case,
the general estimate~\eqref{eq-discrete-general} still formally applies,
but the residual indicator $\mathfrak R_\pi$
may become uninformative because the residual is not explicitly controlled during training.
This explains the lack of rigorous error guarantees for that earlier algorithm.
\end{rem}

\begin{rem}
The constant $C$ in Theorems~\ref{thm-continuous-general} and~\ref{thm-discrete-general} originates solely from the stability analysis of the FBSDE and the time discretization.
It therefore controls the error under the assumption that the discrete approximation $\{(\hat{X}_i,\hat{Y}_i,\hat{Z}_i,\hat{U}_i)\}$ is already available.
When such an approximation is generated by a deep learning algorithm, additional sources of error---primarily the approximation error due to finite network capacity and the statistical error from finite-sample Monte Carlo estimates---are not accounted for in $C$.
Consequently, the a posteriori bound should be interpreted as a computable \emph{error indicator} that faithfully reflects the quality of the discrete trajectory, rather than a certified numerical upper bound in the presence of those unmodeled errors.
Its practical value lies in guiding the training process and comparing different approximations without reference to the true solution, which is exactly how we employ it in Section~\ref{sec-numerics}.
\end{rem}

\section{Proofs of the Main Theorems}
\label{sec-proofs}

In this section we provide the complete and detailed proofs of the two principal theorems stated in Section~\ref{sec-error-analysis}:
the general continuous stability estimate (Theorem~\ref{thm-continuous-general}) and the general a posteriori error bound for discrete approximations (Theorem~\ref{thm-discrete-general}).

\subsection{Proof of Theorem~\ref{thm-continuous-general}}
\label{ssec-proof-continuous-general}

\begin{proof}
Let
\[
    \Delta X_t:=X_t-\cX_t,\qquad
    \Delta Y_t:=Y_t-\cY_t,\qquad
    \Delta Z_t:=Z_t-\cZ_t.
\]
Also define
\[
    \Delta U_t:=Y_t-\cU_t.
\]
Then, by \eqref{eq-FBSDE} and \eqref{eq-perturbed-general},
\begin{equation}\label{eq-delta-system}
    \begin{cases}
        d\Delta X_t
        =
        \bigl(
        b(t,X_t,Y_t,Z_t)-b(t,\cX_t,\cU_t,\cZ_t)-\alpha_t
        \bigr)\,dt\\
        \qquad\qquad
        +
        \bigl(
        \sigma(t,X_t,Y_t,Z_t)-\sigma(t,\cX_t,\cU_t,\cZ_t)-\beta_t
        \bigr)\,dB_t,\\
        -d\Delta Y_t
        =
        \bigl(
        f(t,X_t,Y_t,Z_t)-f(t,\cX_t,\cY_t,\cZ_t)-\gamma_t
        \bigr)\,dt
        -\Delta Z_t\,dB_t,\\
        \Delta X_0=0,\qquad
        \Delta Y_T=g(X_T)-g(\cX_T)+\eta.
    \end{cases}
\end{equation}

For convenience, write
\begin{align*}
    \Delta b_t &:= b(t,X_t,Y_t,Z_t)-b(t,\cX_t,\cY_t,\cZ_t),\\
    \delta b_t &:= b(t,\cX_t,\cY_t,\cZ_t)-b(t,\cX_t,\cU_t,\cZ_t),\\
    \Delta\sigma_t &:= \sigma(t,X_t,Y_t,Z_t)-\sigma(t,\cX_t,\cY_t,\cZ_t),\\
    \delta\sigma_t &:= \sigma(t,\cX_t,\cY_t,\cZ_t)-\sigma(t,\cX_t,\cU_t,\cZ_t),\\
    \Delta f_t &:= f(t,X_t,Y_t,Z_t)-f(t,\cX_t,\cY_t,\cZ_t).
\end{align*}
Then \eqref{eq-delta-system} becomes
\begin{equation}\label{eq-DeltaX-split}
    \begin{cases}
        d\Delta X_t = \bigl(\Delta b_t + \delta b_t - \alpha_t\bigr)\,dt + \bigl(\Delta\sigma_t + \delta\sigma_t - \beta_t\bigr)\,dB_t,\\
        d\Delta Y_t = -\bigl(\Delta f_t - \gamma_t\bigr)\,dt + \Delta Z_t\,dB_t,\\
        \Delta X_0=0,\qquad \Delta Y_T=g(X_T)-g(\cX_T)+\eta.
    \end{cases}
\end{equation}

\medskip
\noindent
\textbf{Forward estimate for $\Delta X$.}
Applying It\^o's formula to $|\Delta X_t|^2$ and using \eqref{eq-DeltaX-split},
\[
    d|\Delta X_t|^2 = 2\langle \Delta X_t, \Delta b_t+\delta b_t-\alpha_t\rangle dt + 2\langle \Delta X_t, \Delta\sigma_t+\delta\sigma_t-\beta_t\rangle dB_t + |\Delta\sigma_t+\delta\sigma_t-\beta_t|^2 dt.
\]
Taking expectations and integrating from $0$ to $t$,
\[
    \E|\Delta X_t|^2 = \E\int_0^t \bigl( 2\langle \Delta X_s, \Delta b_s+\delta b_s-\alpha_s\rangle + |\Delta\sigma_s+\delta\sigma_s-\beta_s|^2 \bigr) ds.
\]
Using the Lipschitz properties of $b$ and $\sigma$, we obtain the pointwise bounds
\begin{align*}
    2\langle \Delta X, \Delta b+\delta b-\alpha\rangle &\le C_1|\Delta X|^2 + |\Delta Y|^2 + |\Delta Z|^2 + |\alpha|^2 + |\cY-\cU|^2,\\
    |\Delta\sigma+\delta\sigma-\beta|^2 &\le C_2\bigl(|\Delta X|^2+|\Delta Y|^2+|\Delta Z|^2+|\beta|^2+|\cY-\cU|^2\bigr).
\end{align*}
Thus
\[
    \E|\Delta X_t|^2 \le C\int_0^t \E|\Delta X_s|^2 ds + C\int_0^T \E\bigl(|\Delta Y_s|^2+|\Delta Z_s|^2\bigr) ds + C\int_0^T \E\bigl(|\alpha_s|^2+|\beta_s|^2+|\cY_s-\cU_s|^2\bigr) ds.
\]
Gronwall's inequality then yields
\begin{equation}\label{eq-estimate-X-sup}
    \sup_{0\le t\le T}\E|\Delta X_t|^2 \le C\Bigl( \E\int_0^T (|\Delta Y_t|^2+|\Delta Z_t|^2) dt + \E\int_0^T (|\alpha_t|^2+|\beta_t|^2+|\cY_t-\cU_t|^2) dt \Bigr).
\end{equation}

\medskip
\noindent
\textbf{Backward estimate for $(\Delta Y,\Delta Z)$.}
Apply It\^o's formula to $|\Delta Y_t|^2$:
\[
    d|\Delta Y_t|^2 = -2\langle \Delta Y_t, \Delta f_t-\gamma_t\rangle dt + 2\langle \Delta Y_t, \Delta Z_t\rangle dB_t + |\Delta Z_t|^2 dt.
\]
Integrating from $t$ to $T$ and taking expectation,
\[
    \E|\Delta Y_t|^2 + \E\int_t^T |\Delta Z_s|^2 ds = \E|\Delta Y_T|^2 + 2\E\int_t^T \langle \Delta Y_s, \Delta f_s-\gamma_s\rangle ds.
\]
The Lipschitz continuity of $f$ gives $|\Delta f_s|\le K(|\Delta X_s|+|\Delta Y_s|+|\Delta Z_s|)$. Then
\[
    2\langle \Delta Y_s, \Delta f_s-\gamma_s\rangle \le K^2|\Delta X_s|^2 + (1+2K+2K^2)|\Delta Y_s|^2 + \frac{1}{2}|\Delta Z_s|^2 + |\gamma_s|^2.
\]
Hence
\[
    \E|\Delta Y_t|^2 + \E\int_t^T |\Delta Z_s|^2 ds \le \E|\Delta Y_T|^2 + C\E\int_t^T \bigl(|\Delta X_s|^2+|\Delta Y_s|^2+|\gamma_s|^2\bigr) ds + \frac{1}{2}\E\int_t^T |\Delta Z_s|^2 ds.
\]
Moving the $\frac{1}{2}\int|\Delta Z|^2$ term to the left,
\begin{align*}
    \E|\Delta Y_t|^2 + \frac{1}{2}\E\int_t^T |\Delta Z_s|^2 ds &\le \E|\Delta Y_T|^2 + C\E\int_t^T \bigl(|\Delta X_s|^2+|\Delta Y_s|^2+|\gamma_s|^2\bigr) ds\\
    &\le\E|\Delta Y_T|^2 + C\E\int_0^T \bigl(|\Delta X_t|^2+|\gamma_t|^2\bigr) dt + C\E\int_t^T |\Delta Y_s|^2 ds.
\end{align*}
Since $g$ is Lipschitz,
$\E|\Delta Y_T|^2 \le C(\E|\Delta X_T|^2 + \E|\eta|^2)$.
Applying Gronwall's inequality to the function $\phi(t)=\E|\Delta Y_t|^2$ (noting that the integral of $|\Delta Y|^2$ can be absorbed) yields
\begin{equation}\label{eq-estimate-Y-sup}
    \sup_{0\le t\le T}\E|\Delta Y_t|^2 \le C\Bigl( \E\int_0^T (|\Delta X_t|^2+|\gamma_t|^2) dt + \E|\Delta X_T|^2 + \E|\eta|^2 \Bigr).
\end{equation}
Together with the bound for $\int|\Delta Z|^2$ obtained from the same inequality,
we get
\begin{equation}\label{eq-estimate-YZ-sup}
    \sup_{0\le t\le T}\E|\Delta Y_t|^2 + \E\int_0^T |\Delta Z_t|^2 dt \le C\Bigl( \E\int_0^T (|\Delta X_t|^2+|\gamma_t|^2) dt + \E|\Delta X_T|^2 + \E|\eta|^2 \Bigr).
\end{equation}

\medskip
For $u=(x,y,z)$, recall
\[
    A(t,u):=
    \begin{pmatrix}
        -G^\top f(t,x,y,z)\\Gb(t,x,y,z)\\G\sigma(t,x,y,z)
    \end{pmatrix}.
\]
Applying It\^o's formula to $\langle G\Delta X_t,\Delta Y_t\rangle$ gives
\begin{align*}
    d\langle G\Delta X_t,\Delta Y_t\rangle &= \langle G(\Delta b_t+\delta b_t-\alpha_t),\Delta Y_t\rangle dt + \langle G(\Delta\sigma_t+\delta\sigma_t-\beta_t),\Delta Y_t\rangle dB_t \\
    &\quad + \langle G\Delta X_t,-(\Delta f_t-\gamma_t)\rangle dt + \langle G\Delta X_t,\Delta Z_t\rangle dB_t \\
    &\quad + \langle G(\Delta\sigma_t+\delta\sigma_t-\beta_t),\Delta Z_t\rangle dt.
\end{align*}
Integrating from $0$ to $T$ and taking expectation (the stochastic integrals are martingales and vanish),
we obtain
\begin{equation}\label{eq-energy-compact}
    \begin{aligned}
        \E\bigl[\langle G\Delta X_T,\Delta Y_T\rangle\bigr]
        &= \E\Bigl[\int_0^T \langle A(t,u_t)-A(t,\bar u_t),u_t-\bar u_t\rangle dt\Bigr] \\
        &\quad + \E\Bigl[\int_0^T \langle G\delta b_t,\Delta Y_t\rangle dt\Bigr]
        + \E\Bigl[\int_0^T \langle G\delta\sigma_t,\Delta Z_t\rangle dt\Bigr] \\
        &\quad + \E\Bigl[\int_0^T \bigl( \langle G\Delta X_t,\gamma_t\rangle - \langle G\alpha_t,\Delta Y_t\rangle - \langle G\beta_t,\Delta Z_t\rangle \bigr)dt\Bigr],
    \end{aligned}
\end{equation}
where
\[
    u_t:=(X_t,Y_t,Z_t),\qquad \bar u_t:=(\cX_t,\cY_t,\cZ_t),
\]
and
\[
    \langle A(t,u_t)-A(t,\bar u_t),u_t-\bar u_t\rangle
    = -\langle G\Delta X_t,\Delta f_t\rangle + \langle G\Delta b_t,\Delta Y_t\rangle + \langle G\Delta\sigma_t,\Delta Z_t\rangle.
\]

\medskip
By the Lipschitz continuity of $b$ and $\sigma$ (Assumption~\ref{assu-1}),
\begin{equation}\label{eq-deltab-est}
    |\delta b_t|
    = |b(t,\cX_t,\cY_t,\cZ_t)-b(t,\cX_t,\cU_t,\cZ_t)|
    \le K|\cY_t-\cU_t|,
\end{equation}
\begin{equation}\label{eq-deltasigma-est}
    |\delta\sigma_t|
    = |\sigma(t,\cX_t,\cY_t,\cZ_t)-\sigma(t,\cX_t,\cU_t,\cZ_t)|
    \le K|\cY_t-\cU_t|.
\end{equation}
Moreover, by the monotonicity condition in Assumption~\ref{assu-1},
\begin{equation}\label{eq-monotonicity}
    \langle A(t,u_t)-A(t,\bar u_t),u_t-\bar u_t\rangle
    \le -\beta_1 |G\Delta X_t|^2
    - \beta_2\bigl( |G^\top\Delta Y_t|^2 + |G^\top\Delta Z_t|^2 \bigr).
\end{equation}

We will repeatedly use Young's inequality $2ab \le \varepsilon a^2 + \varepsilon^{-1}b^2$ with different small parameters $\varepsilon_1$ (for terms involving $|G\Delta X|^2$) and $\varepsilon_2$ (for $|\Delta Y|^2,|\Delta Z|^2$).
The precise values of $\varepsilon_1,\varepsilon_2$ will be chosen later sufficiently small to absorb certain terms.
Constants $C_{\varepsilon}$ depend on $\varepsilon$ and may change from line to line.

From \eqref{eq-deltab-est} and \eqref{eq-deltasigma-est} we obtain
\begin{align*}
    |\langle G\delta b_t,\Delta Y_t\rangle| &\le \varepsilon_2 |\Delta Y_t|^2 + C_{\varepsilon_2} |\cY_t-\cU_t|^2,\\
    |\langle G\delta\sigma_t,\Delta Z_t\rangle| &\le \varepsilon_2 |\Delta Z_t|^2 + C_{\varepsilon_2} |\cY_t-\cU_t|^2.
\end{align*}
Hence
\begin{equation}\label{eq-error-terms}
    |\langle G\delta b_t,\Delta Y_t\rangle|+|\langle G\delta\sigma_t,\Delta Z_t\rangle|
    \le \varepsilon_2(|\Delta Y_t|^2+|\Delta Z_t|^2) + C_{\varepsilon_2} |\cY_t-\cU_t|^2.
\end{equation}
Similarly,
\begin{align*}
    |\langle G\alpha_t,\Delta Y_t\rangle| &\le \varepsilon_2 |\Delta Y_t|^2 + C_{\varepsilon_2}|\alpha_t|^2,\\
    |\langle G\beta_t,\Delta Z_t\rangle| &\le \varepsilon_2 |\Delta Z_t|^2 + C_{\varepsilon_2}|\beta_t|^2,\\
    |\langle G\Delta X_t,\gamma_t\rangle| &\le \varepsilon_1 |G\Delta X_t|^2 + C_{\varepsilon_1}|\gamma_t|^2.
\end{align*}
Thus
\begin{equation}\label{eq-perturbation-terms}
    \begin{aligned}
        &\bigl| \langle G\alpha_t,\Delta Y_t\rangle - \langle G\Delta X_t,\gamma_t\rangle + \langle G\beta_t,\Delta Z_t\rangle \bigr| \\
        &\qquad\le \varepsilon_1 |G\Delta X_t|^2 + \varepsilon_2(|\Delta Y_t|^2+|\Delta Z_t|^2)
        + C_{\varepsilon_1}|\gamma_t|^2 + C_{\varepsilon_2}(|\alpha_t|^2+|\beta_t|^2).
    \end{aligned}
\end{equation}

Inserting the monotonicity bound \eqref{eq-monotonicity} and the estimates \eqref{eq-error-terms}, \eqref{eq-perturbation-terms} into \eqref{eq-energy-compact} yields
\begin{equation}\label{eq-estimate-base}
    \begin{aligned}
        &\E\langle G\Delta X_T,\Delta Y_T\rangle + \beta_1\E\int_0^T |G\Delta X_t|^2 dt + \beta_2\E\int_0^T \bigl(|G^\top\Delta Y_t|^2+|G^\top\Delta Z_t|^2\bigr) dt \\
        &\le \E\int_0^T \bigl( \varepsilon_1|G\Delta X_t|^2 + 2\varepsilon_2(|\Delta Y_t|^2+|\Delta Z_t|^2) \bigr) dt \\
        &\quad + \E\int_0^T \Bigl( C_{\varepsilon_2}(|\alpha_t|^2+|\beta_t|^2+|\cY_t-\cU_t|^2) + C_{\varepsilon_1}|\gamma_t|^2 \Bigr) dt.
    \end{aligned}
\end{equation}
By Assumption~\ref{assu-1},
the terminal condition satisfies
\[
    \langle G\Delta X_T, g(X_T)-g(\cX_T)\rangle \ge \mu |G\Delta X_T|^2.
\]
Using $\Delta Y_T = g(X_T)-g(\cX_T)+\eta$, we have
\[
    \langle G\Delta X_T,\Delta Y_T\rangle = \langle G\Delta X_T, g(X_T)-g(\cX_T)\rangle + \langle G\Delta X_T,\eta\rangle
    \ge \mu |G\Delta X_T|^2 - |\langle G\Delta X_T,\eta\rangle|.
\]
Substituting this into \eqref{eq-estimate-base} and rearranging, we obtain
\begin{equation}\label{eq-estimate-total}
    \begin{aligned}
        &\mu\E|G\Delta X_T|^2 + \beta_1\E\int_0^T |G\Delta X_t|^2 dt + \beta_2\E\int_0^T \bigl(|G^\top\Delta Y_t|^2+|G^\top\Delta Z_t|^2\bigr) dt \\
        &\le \E |\langle G\Delta X_T,\eta\rangle|
        + \E\int_0^T \bigl( \varepsilon_1|G\Delta X_t|^2 + 2\varepsilon_2(|\Delta Y_t|^2+|\Delta Z_t|^2) \bigr) dt \\
        &\quad + \E\int_0^T \Bigl( C_{\varepsilon_2}(|\alpha_t|^2+|\beta_t|^2+|\cY_t-\cU_t|^2) + C_{\varepsilon_1}|\gamma_t|^2 \Bigr) dt.
    \end{aligned}
\end{equation}

\medskip
\noindent
\textbf{Case 1: $m>n$ (the $X$-coercive regime).}
Here $G$ has full column rank, so there exists $c_G>0$ such that $|Gx|\ge c_G|x|$ for all $x\in\R^n$.
In particular, $|G\Delta X|^2\ge c_G^2|\Delta X|^2$ and $|G\Delta X_T|^2\ge c_G^2|\Delta X_T|^2$.

Young's inequality gives
\[
    |\langle G\Delta X_T,\eta\rangle| \le \frac{\mu}{2}|G\Delta X_T|^2 + \frac{1}{2\mu}|\eta|^2.
\]
Taking expectations, substituting it into \eqref{eq-estimate-base} and rearranging terms,
we obtain
\begin{equation}\label{eq-estimate-total-1}
    \begin{aligned}
        &\frac{\mu}{2}\E|G\Delta X_T|^2 + \beta_1\E\int_0^T |G\Delta X_t|^2 dt + \beta_2\E\int_0^T \bigl(|G^\top\Delta Y_t|^2+|G^\top\Delta Z_t|^2\bigr) dt \\
        &\le \frac{1}{2\mu}\E|\eta|^2 
        + \E\int_0^T \bigl( \varepsilon_1|G\Delta X_t|^2 + 2\varepsilon_2(|\Delta Y_t|^2+|\Delta Z_t|^2) \bigr) dt \\
        &\quad + \E\int_0^T \Bigl( C_{\varepsilon_2}(|\alpha_t|^2+|\beta_t|^2+|\cY_t-\cU_t|^2) + C_{\varepsilon_1}|\gamma_t|^2 \Bigr) dt.
    \end{aligned}
\end{equation}

Choose $\varepsilon_1 = \beta_1/2$ in \eqref{eq-estimate-total-1}.
Then the terms $\varepsilon_1\int|G\Delta X|^2$ on the right can be absorbed into the left-hand side $\beta_1\int|G\Delta X|^2$.
Dropping the non‑negative $\beta_2$ term, we obtain
\begin{align*}
    \frac{\mu}{2}\E|G\Delta X_T|^2 + \frac{\beta_1}{2}\E\int_0^T |G\Delta X_t|^2 dt &\le C_{\varepsilon_1}\E|\eta|^2 + 2\varepsilon_2\E\int_0^T (|\Delta Y_t|^2+|\Delta Z_t|^2) dt \\
    &+ C_{\varepsilon_2}\E\int_0^T (|\alpha_t|^2+|\beta_t|^2+|\cY_t-\cU_t|^2) dt + C_{\varepsilon_1}\E\int_0^T |\gamma_t|^2 dt.
\end{align*}
The left-hand side dominates $\frac{\mu c_G^2}{2}\E|\Delta X_T|^2 + \frac{\beta_1 c_G^2}{2}\E\int_0^T |\Delta X_t|^2 dt$ from the full column rank property.
Consequently,
\begin{align*}
    \E\int_0^T |\Delta X_t|^2 dt + \E|\Delta X_T|^2 & \le C\Bigl( \E|\eta|^2 + \varepsilon_2\E\int_0^T (|\Delta Y_t|^2+|\Delta Z_t|^2) dt \\
    &\qquad+ \E\int_0^T (|\alpha_t|^2+|\beta_t|^2+|\gamma_t|^2+|\cY_t-\cU_t|^2) dt \Bigr).
\end{align*}
Note that \eqref{eq-estimate-YZ-sup} holds,
we have
\begin{align*}
    & \sup_{0 \le t \le T} \E \bigl[|\Delta Y_t|^2\bigr] + \E\Bigl[\int_0^T |\Delta Z_t|^2 dt\Bigr] \\
    \le &C \Bigl(\E \Bigl[\int_0^T |\Delta X_t|^2dt\Bigr] + \E \bigl[|\Delta X_T|^2\bigr] + \E \Bigl[\int_0^T|\gamma_t|^2dt\Bigr] + \E \bigl[|\eta|^2\bigr]\Bigr) \\
    \le &C\Bigl(\E \Bigl[\int_0^T\varepsilon_2\bigl(|\Delta Y_t|^2 + |\Delta Z_t|^2 \bigr) dt\Bigr] + \E \bigl[|\eta|^2\bigr] + \E \Bigl[\int_0^T\bigl(|\alpha_t|^2 + |\beta_t|^2 + |\gamma_t|^2 \bigr) dt\Bigr] \\
    & \qquad+ \E \Bigl[\int_0^T|\cY_t-\cU_t|^2 dt\Bigr]\Bigr).
\end{align*}
Choosing a small $\varepsilon_2$ (e.g., $\varepsilon_2 = \frac{1}{2C(1+T)}$),
we get the estimate of $Y,Z$
\begin{align*}
    & \sup_{0 \le t \le T} \E \bigl[|\Delta Y_t|^2\bigr] + \E\Bigl[\int_0^T |\Delta Z_t|^2 dt\Bigr] \\
    \le &C\Bigl(\E \bigl[|\eta|^2\bigr] + \E \Bigl[\int_0^T\bigl(|\alpha_t|^2 + |\beta_t|^2 + |\gamma_t|^2 \bigr) dt\Bigr] + \E \Bigl[\int_0^T|\cY_t-\cU_t|^2 dt\Bigr]\Bigr).
\end{align*}
Substituting it into \eqref{eq-estimate-X-sup} can derive the estimation of $X$
\begin{align*}
    \sup_{0 \le t \le T} \E \bigl[|\Delta X_t|^2\bigr]
    \le C\Bigl(\E \bigl[|\eta|^2\bigr] + \E \Bigl[\int_0^T\bigl(|\alpha_t|^2 + |\beta_t|^2 + |\gamma_t|^2 \bigr) dt\Bigr] + \E \Bigl[\int_0^T|\cY_t-\cU_t|^2 dt\Bigr]\Bigr).
\end{align*}
Thus the desired estimate holds in Case 1.

\medskip
\noindent
\textbf{Case 2: $m < n$ (the $(Y,Z)$-coercive regime).}
In this case $G$ has full row rank,
so there exists $c_G'>0$ such that $|G^\top y|\ge c_G'|y|$ for all $y\in\R^m$, and the same holds columnwise for matrices.
Hence $|G^\top\Delta Y|^2\ge (c_G')^2|\Delta Y|^2$ and $|G^\top\Delta Z|^2\ge (c_G')^2|\Delta Z|^2$.

Similarly,
Young's inequality gives
\[
    |\langle G\Delta X_T,\eta\rangle| \le \varepsilon_1|G\Delta X_T|^2 + \frac{1}{4\varepsilon_1}|\eta|^2.
\]
Thus, we will obtain from \eqref{eq-estimate-base}
\begin{equation}\label{eq-estimate-total-2}
    \begin{aligned}
        &\mu\E|G\Delta X_T|^2 + \beta_1\E\int_0^T |G\Delta X_t|^2 dt + \beta_2\E\int_0^T \bigl(|G^\top\Delta Y_t|^2+|G^\top\Delta Z_t|^2\bigr) dt \\
        &\le C_{\varepsilon_1}\E|\eta|^2 + \varepsilon_1 \E|G\Delta X_T|^2
        + \E\int_0^T \bigl( \varepsilon_1|G\Delta X_t|^2 + 2\varepsilon_2(|\Delta Y_t|^2+|\Delta Z_t|^2) \bigr) dt \\
        &\quad + \E\int_0^T \Bigl( C_{\varepsilon_2}(|\alpha_t|^2+|\beta_t|^2+|\cY_t-\cU_t|^2) + C_{\varepsilon_1}|\gamma_t|^2 \Bigr) dt.
    \end{aligned}
\end{equation}

Choose $\varepsilon_2 = \beta_2(c_G')^2/4$ in \eqref{eq-estimate-total-2}.
Dropping the non‑negative $\mu\E|G\Delta X_T|^2$ and $\beta_1\int|G\Delta X|^2$  and rearranging terms,
we obtain
\begin{align*}
    \frac{\beta_2(c_G')^2}{2}\E\int_0^T (|\Delta Y_t|^2+|\Delta Z_t|^2) dt &\le C_{\varepsilon_1}\E|\eta|^2 + \varepsilon_1\E|G\Delta X_T|^2+ \varepsilon_1\E\int_0^T |G\Delta X_t|^2 dt \\
    &+ C_{\varepsilon_2}\E\int_0^T (|\alpha_t|^2+|\beta_t|^2+|\cY_t-\cU_t|^2) dt + C_{\varepsilon_1}\E\int_0^T |\gamma_t|^2 dt.
\end{align*}
Now use the forward estimate \eqref{eq-estimate-X-sup} (note $|G\Delta X|^2\le \|G\|^2|\Delta X|^2$),
we have
\[
    \E |G\Delta X_T|^2 + \E\int_0^T |G\Delta X_t|^2 dt \le C\Bigl( \E\int_0^T (|\Delta Y_t|^2+|\Delta Z_t|^2) dt + \E\int_0^T (|\alpha_t|^2+|\beta_t|^2+|\cY_t-\cU_t|^2) dt \Bigr).
\]
Substituting this into the previous inequality gives
\[
    \E\int_0^T (|\Delta Y_t|^2+|\Delta Z_t|^2) dt \le C\Bigl( \E|\eta|^2 + \varepsilon_1\E\int_0^T (|\Delta Y|^2+|\Delta Z|^2) + \E\int_0^T (|\alpha|^2+|\beta|^2+|\gamma|^2+|\cY-\cU|^2) \Bigr).
\]
Choosing $\varepsilon_1$ sufficiently small (e.g., $\varepsilon_1 = \frac{1}{2C}$) allows us to absorb the $\varepsilon_1$ term,
yielding
\begin{equation}\label{eq-estimate-YZ}
    \E\int_0^T (|\Delta Y_t|^2+|\Delta Z_t|^2) dt \le C\Bigl( \E|\eta|^2 + \E\int_0^T (|\alpha_t|^2+|\beta_t|^2+|\gamma_t|^2+|\cY_t-\cU_t|^2) dt \Bigr).
\end{equation}
Plugging this bound into \eqref{eq-estimate-X-sup} gives the same estimate for $|\Delta X_t|^2$.
\[
    \sup_{0\le t\le T} \E |\Delta X_t|^2 \le C\Bigl( \E|\eta|^2 + \E\int_0^T (|\alpha_t|^2+|\beta_t|^2+|\gamma_t|^2+|\cY_t-\cU_t|^2) dt \Bigr).
\]
Finally, \eqref{eq-estimate-YZ-sup} together with \eqref{eq-estimate-YZ} and $\E|\Delta X_T|^2\le\sup_{0\le t\le T}\E|\Delta X_t|^2$ yields
\[
    \sup_{0\le t\le T}\E|\Delta Y_t|^2 \le C\Bigl( \E|\eta|^2 + \E\int_0^T (|\alpha_t|^2+|\beta_t|^2+|\gamma_t|^2+|\cY_t-\cU_t|^2) dt \Bigr).
\]
Thus the desired estimate also holds in Case 2.

\medskip
\noindent
\textbf{Case 3: $m=n$ (square case).}
When $m=n$, the matrix $G$ is invertible. If $\beta_2>0$, the argument of Case 2 applies; if $\beta_2=0$, then Assumption~\ref{assu-1} forces $\beta_1>0$ and $\mu>0$, and Case 1 applies. Hence the estimate follows.

\medskip
\noindent
Combining all cases completes the proof of Theorem~\ref{thm-continuous-general}.
\end{proof}

\subsection{Proof of Theorem~\ref{thm-discrete-general}}

\begin{proof}
We only prove \eqref{eq-discrete-general}, since \eqref{eq-discrete-general-grid}
follows immediately from \eqref{eq-discrete-general} by evaluating the estimate
at the grid points.

For each $i=0,\ldots,N-1$, set
\[
    b_i:=b(t_i,\hat X_i,\hat U_i,\hat Z_i),\qquad
    \sigma_i:=\sigma(t_i,\hat X_i,\hat U_i,\hat Z_i),\qquad
    f_i:=f(t_i,\hat X_i,\hat Y_i,\hat Z_i).
\]
By the martingale representation theorem, there exist progressively measurable
processes $\rho^{X,i}$ and $\rho^{Y,i}$ on $[t_i,t_{i+1}]$ such that
\[
    R_i^X
    =
    \E_{t_i}[R_i^X]
    +
    \int_{t_i}^{t_{i+1}}\rho_s^{X,i}\,dB_s,
\]
and
\[
    R_i^Y
    =
    \E_{t_i}[R_i^Y]
    +
    \int_{t_i}^{t_{i+1}}\rho_s^{Y,i}\,dB_s.
\]
Moreover,
\[
    \E\int_{t_i}^{t_{i+1}}|\rho_s^{X,i}|^2\,ds
    \le
    \E|R_i^X|^2,
    \qquad
    \E\int_{t_i}^{t_{i+1}}|\rho_s^{Y,i}|^2\,ds
    \le
    \E|R_i^Y|^2.
\]

Define, for $t\in[t_i,t_{i+1}]$,
\[
    \cX_t
    :=
    \hat X_i
    +
    \int_{t_i}^{t}
    \left(
        b_i+\frac{\E_{t_i}[R_i^X]}{\Delta t_i}
    \right)ds
    +
    \int_{t_i}^{t}
    \left(
        \sigma_i+\rho_s^{X,i}
    \right)dB_s,
\]
and
\[
    \cY_t
    :=
    \hat Y_i
    +
    \int_{t_i}^{t}
    \left(
        -f_i+\frac{\E_{t_i}[R_i^Y]}{\Delta t_i}
    \right)ds
    +
    \int_{t_i}^{t}
    \left(
        \hat Z_i+\rho_s^{Y,i}
    \right)dB_s.
\]
Set
\[
    \cZ_t:=\hat Z_i+\rho_t^{Y,i},\qquad
    \cU_t:=\hat U_i,\qquad t\in[t_i,t_{i+1}].
\]
Then, by the definitions of $R_i^X$ and $R_i^Y$, we have
\[
    \cX_{t_i}=\hat X_i,\qquad
    \cX_{t_{i+1}}=\hat X_{i+1},
\]
and
\[
    \cY_{t_i}=\hat Y_i,\qquad
    \cY_{t_{i+1}}=\hat Y_{i+1}.
\]
In particular,
\[
    \cX_0=a,\qquad
    \cY_T=\hat Y_N
    =
    g(\hat X_N)-\eta
    =
    g(\cX_T)-\eta .
\]
Equivalently,
\[
    Y_T-\cY_T
    =
    g(X_T)-g(\cX_T)+\eta .
\]

We now write $(\cX,\cY,\cZ,\cU)$ as a perturbed
continuous-time system. On $t\in[t_i,t_{i+1}]$, define
\[
    \alpha_t
    :=
    b_i+\frac{\E_{t_i}[R_i^X]}{\Delta t_i}
    -
    b(t,\cX_t,\cU_t,\cZ_t),
\]
\[
    \beta_t
    :=
    \sigma_i+\rho_t^{X,i}
    -
    \sigma(t,\cX_t,\cU_t,\cZ_t),
\]
and
\[
    \gamma_t
    :=
    f_i
    - f(t,\cX_t,\cY_t,\cZ_t) - 
    \frac{\E_{t_i}[R_i^Y]}{\Delta t_i}.
\]
Then $(\cX,\cY,\cZ)$ satisfies
\[
    d\cX_t
    =
    \bigl[
        b(t,\cX_t,\cU_t,\cZ_t)+\alpha_t
    \bigr]dt
    +
    \bigl[
        \sigma(t,\cX_t,\cU_t,\cZ_t)+\beta_t
    \bigr]dB_t,
\]
and
\[
    -d\cY_t
    =
    \bigl[
        f(t,\cX_t,\cY_t,\cZ_t)+\gamma_t
    \bigr]dt
    -
    \cZ_t\,dB_t,
\]
with terminal mismatch $\eta$. Hence Theorem~\ref{thm-continuous-general}
can be applied to compare $(X,Y,Z)$ and
$(\cX,\cY,\cZ)$.

It remains to estimate the perturbation terms. By the Lipschitz continuity in
$(x,y,z)$, the $1/2$-H\"older continuity in time, and the definitions above,
for $t\in[t_i,t_{i+1}]$,
\[
\begin{aligned}
|\alpha_t|^2
&\le
C\Bigl(
    |t-t_i|
    +
    |\cX_t-\hat X_i|^2
    +
    |\cZ_t-\hat Z_i|^2
    +
    \frac{|\E_{t_i}[R_i^X]|^2}{(\Delta t_i)^2}
\Bigr),
\\
|\beta_t|^2
&\le
C\Bigl(
    |t-t_i|
    +
    |\cX_t-\hat X_i|^2
    +
    |\cZ_t-\hat Z_i|^2
    +
    |\rho_t^{X,i}|^2
\Bigr),
\\
|\gamma_t|^2
&\le
C\Bigl(
    |t-t_i|
    +
    |\cX_t-\hat X_i|^2
    +
    |\cY_t-\hat Y_i|^2
    +
    |\cZ_t-\hat Z_i|^2
    +
    \frac{|\E_{t_i}[R_i^Y]|^2}{(\Delta t_i)^2}
\Bigr).
\end{aligned}
\]
Furthermore,
\[
    \cZ_t-\hat Z_i=\rho_t^{Y,i}.
\]
By Assumption~\ref{assu-2},
together with the linear growth of $b,\sigma,f$ (which follows from their Lipschitz continuity),
this implies
\[
\sup_i \E[|b_i|^2 + |\sigma_i|^2 + |f_i|^2] \le C.
\]
Using these bounds and the construction of $\cX_t$, $\cY_t$,
a standard computation yields
\[
    \sum_{i=0}^{N-1}
    \E\int_{t_i}^{t_{i+1}}
    |\cX_t-\hat X_i|^2\,dt
    \le
    C\bigl(|\pi|+\mathfrak R_\pi\bigr),
\]
and
\[
    \sum_{i=0}^{N-1}
    \E\int_{t_i}^{t_{i+1}}
    |\cY_t-\hat Y_i|^2\,dt
    \le
    C\bigl(|\pi|+\mathfrak R_\pi\bigr).
\]
Also,
\[
    \sum_{i=0}^{N-1}
    \E\int_{t_i}^{t_{i+1}}
    |\rho_t^{X,i}|^2\,dt
    \le
    \sum_{i=0}^{N-1}\E|R_i^X|^2
    \le
    |\pi|\mathfrak R_\pi
    \le
    T\mathfrak R_\pi,
\]
and similarly,
\[
    \sum_{i=0}^{N-1}
    \E\int_{t_i}^{t_{i+1}}
    |\rho_t^{Y,i}|^2\,dt
    \le
    T\mathfrak R_\pi.
\]
Moreover,
\[
    \sum_{i=0}^{N-1}
    \E\int_{t_i}^{t_{i+1}}
    \frac{|\E_{t_i}[R_i^X]|^2}{(\Delta t_i)^2}\,dt
    \le
    \sum_{i=0}^{N-1}
    \E\frac{|R_i^X|^2}{\Delta t_i},
\]
and
\[
    \sum_{i=0}^{N-1}
    \E\int_{t_i}^{t_{i+1}}
    \frac{|\E_{t_i}[R_i^Y]|^2}{(\Delta t_i)^2}\,dt
    \le
    \sum_{i=0}^{N-1}
    \E\frac{|R_i^Y|^2}{\Delta t_i}.
\]
Combining the preceding estimates gives
\[
    \E\int_0^T
    \bigl(
        |\alpha_t|^2+|\beta_t|^2+|\gamma_t|^2
    \bigr)dt
    \le
    C\bigl(\mathfrak R_\pi+|\pi|\bigr).
\]
Moreover, since $\cU_t=\hat U_i$ and
$\bar Y_t=\hat Y_i$ on $[t_i,t_{i+1})$,
\[
\begin{aligned}
    \E\int_0^T|\cY_t-\cU_t|^2dt
    &\le
    C\E\int_0^T|\cY_t-\bar Y_t|^2dt
    +
    C\E\int_0^T|\bar Y_t-\bar U_t|^2dt
    \\
    &\le
    C\bigl(|\pi|+\mathfrak R_\pi\bigr)
    +
    C\mathfrak P_\pi .
\end{aligned}
\]

Applying Theorem~\ref{thm-continuous-general}, we therefore obtain
\[
\begin{aligned}
    \sup_{0\le t\le T}
    \E\bigl[
        |X_t-\cX_t|^2
        +
        |Y_t-\cY_t|^2
    \bigr]
    +
    \E\int_0^T|Z_t-\cZ_t|^2dt
    \le
    C\bigl(
        \E|\eta|^2
        +
        \mathfrak R_\pi
        +
        \mathfrak P_\pi
        +
        |\pi|
    \bigr).
\end{aligned}
\]

It remains to pass from the continuous interpolation
$(\cX,\cY,\cZ)$ to the piecewise constant interpolation
$(\bar X,\bar Y,\bar Z)$. From the estimates above,
\[
    \sup_{0\le t\le T}
    \E|\cX_t-\bar X_t|^2
    +
    \sup_{0\le t\le T}
    \E|\cY_t-\bar Y_t|^2
    \le
    C\bigl(|\pi|+\mathfrak R_\pi\bigr),
\]
and
\[
    \E\int_0^T|\cZ_t-\bar Z_t|^2dt
    =
    \sum_{i=0}^{N-1}
    \E\int_{t_i}^{t_{i+1}}|\rho_t^{Y,i}|^2dt
    \le
    C\mathfrak R_\pi.
\]
Consequently, by the triangle inequality,
\[
\begin{aligned}
    &\sup_{0\le t\le T}
    \E\bigl[
        |X_t-\bar X_t|^2
        +
        |Y_t-\bar Y_t|^2
    \bigr]
    +
    \E\int_0^T|Z_t-\bar Z_t|^2dt
    \\
    &\qquad\le
    C\bigl(
        \E|\eta|^2
        +
        \mathfrak R_\pi
        +
        \mathfrak P_\pi
        +
        |\pi|
    \bigr).
\end{aligned}
\]
This proves \eqref{eq-discrete-general}. The grid-point estimate
\eqref{eq-discrete-general-grid} follows since
$\bar X_{t_i}=\hat X_i$ and $\bar Y_{t_i}=\hat Y_i$.
The proof is complete.
\end{proof}

\section{Numerical Experiments}
\label{sec-numerics}

In this section we present numerical experiments designed to illustrate the a posteriori error estimates developed in Section~\ref{sec-error-analysis}.
The purpose of the experiments is not to prove convergence rates for a particular neural network architecture,
but rather to examine whether the computable quantities appearing in Theorem~\ref{thm-discrete-general},
namely
\[
    \E[|\eta|^2],\qquad \mathfrak R_\pi,\qquad \mathfrak P_\pi,
\]
provide meaningful diagnostic information for deep learning approximations of fully coupled FBSDEs.

The experiments are organized as follows.
In the first example,
an explicit solution is available through a Riccati equation.
This allows us to compare the neural network approximation with a reference solution and to check whether the a posteriori indicators are consistent with the observed approximation error.
In the second example,
we consider a Burgers-type fully coupled FBSDE for which no closed-form solution is used.
This example is intended to test whether the three components of the a posteriori loss can serve as computable diagnostics and whether removing some of them leads to less stable or less consistent numerical approximations.

\subsection{Network Architecture and Training Configuration}
\label{ssec-exp-setup}

We now describe how the abstract discrete approximation appearing in
Theorem~\ref{thm-discrete-general} is realized by neural network parameterizations.
Let
\[
    \pi=\{0=t_0<t_1<\cdots<t_N=T\},\qquad \Delta t=\frac{T}{N},
\]
be a uniform partition of the time interval.
The numerical method constructs a discrete adapted trajectory
\[
    \{(\hat X_i,\hat Y_i,\hat Z_i,\hat U_i)\}_{i=0}^{N-1}, \qquad (\hat X_N,\hat Y_N),
\]
where the auxiliary process $\hat U_i$ is used in the forward coefficients and is allowed to differ from $\hat Y_i$.

We introduce three families of feedforward neural networks:
\begin{align}
    \cY_i(\cdot;\theta_i^Y) &: \R^n\to\R^m, \qquad i=0,1,\ldots,N, \label{net-Y}\\
    \cZ_i(\cdot;\theta_i^Z) &: \R^n\to\R^{m\times d}, \qquad i=0,1,\ldots,N-1, \label{net-Z}\\
    \cU_i(\cdot;\theta_i^U) &: \R^n\to\R^m, \qquad i=0,1,\ldots,N-1. \label{net-U}
\end{align}
Here the networks $\{\cZ_i\}$ and $\{\cU_i\}$ serve as standard approximations of the control process and the auxiliary control entering the forward dynamics,
respectively.  
The backward component is, however, constructed recursively via \emph{residual networks}.
Specifically, $\cY_0$ acts as an initial value network,
while for $i\ge 1$,
each $\cY_i$ directly outputs the one-step backward residual $R_{i-1}^Y$ defined in~\eqref{eq-res-Y}.
This parameterization exposes the pathwise residual exactly where it is needed for the a posteriori loss.

Concretely, given simulated Brownian increments $\Delta B_i$,
the discrete trajectory is generated by
\begin{align}
    \hat Y_0 &= \cY_0(\hat X_0;\theta_0^Y), \label{eq-num-Y0}\\
    \hat Z_i &= \cZ_i(\hat X_i;\theta_i^Z),\qquad
    \hat U_i = \cU_i(\hat X_i;\theta_i^U). \label{eq-num-ZU}\\
    \hat X_{i+1} &= \hat X_i + b(t_i,\hat X_i,\hat U_i,\hat Z_i)\Delta t
                  + \sigma(t_i,\hat X_i,\hat U_i,\hat Z_i)\Delta B_i, \label{eq-num-X}\\
    \hat Y_{i+1} &= \hat Y_i - f(t_i,\hat X_i,\hat Y_i,\hat Z_i)\Delta t
                  + \hat Z_i\,\Delta B_i + \cY_{i+1}(\hat X_{i+1};\theta_{i+1}^Y). \label{eq-num-Y}
\end{align}
Comparing \eqref{eq-num-Y} with \eqref{eq-res-Y} shows that
\[
    R_i^Y = \cY_{i+1}(\hat X_{i+1};\theta_{i+1}^Y),
\]
i.e., the residual network $\cY_{i+1}$ directly learns the defect in the
backward dynamics.
The terminal defect is
\[
    \eta = g(\hat X_N) - \hat Y_N .
\]
The empirical training loss is chosen in accordance with Theorem~\ref{thm-discrete-general} and takes the three-component form
\begin{equation}\label{eq-training-loss}
    \mathcal L = \mathcal L_T + \lambda_R \mathcal L_R + \lambda_U \mathcal L_U,
\end{equation}
where
\begin{align}
    \mathcal L_T &= \widehat{\E}\bigl[|g(\hat X_N)-\hat Y_N|^2\bigr],\\
    \mathcal L_R &= \sum_{i=0}^{N-1}\widehat{\E}\biggl[\frac{|\cY_{i+1}(\hat X_{i+1})|^2}{\Delta t}\biggr]
                 = \sum_{i=0}^{N-1}\widehat{\E}\biggl[\frac{|R_i^Y|^2}{\Delta t}\biggr],\\
    \mathcal L_U &= \sum_{i=0}^{N-1}
    \widehat{\E}\bigl[|\hat Y_i-\hat U_i|^2\bigr]\Delta t .
\end{align}
The constants $\lambda_R$ and $\lambda_U$ are penalty weights;
they are not to be confused with the coupling parameters of the FBSDE itself.
The loss \eqref{eq-training-loss} is a Monte Carlo approximation of the computable terms appearing in the a posteriori bound,
up to the deterministic discretization term $|\pi|$.

The complete training procedure, which integrates the forward simulation,
the computation of the three-component loss \eqref{eq-training-loss},
and the parameter updates,
is summarized in Algorithm~\ref{alg:training}.
At each iteration, a batch of Brownian motion paths is generated,
the forward SDE is solved using the current networks $\{\mathcal{Y}_i, \mathcal{Z}_i, \mathcal{U}_i\}$,
and the empirical loss $\mathcal{L}$ is evaluated by Monte Carlo averaging.
All network parameters are then updated by a stochastic gradient descent variant (Adam).
The algorithm corresponds exactly to the discrete scheme described in Section~\ref{ssec-disc-general},
and every term in the loss function mirrors a computable component of the a posteriori bound in Theorem~\ref{thm-discrete-general}.

\begin{algorithm}[tb]
\caption{Deep FBSDE with control mismatch}
\label{alg:training}
\begin{algorithmic}
\State \textbf{Input:} Time grid $\{t_i\}_{i=0}^N$, number of sample paths $M$, penalty weights $\lambda_R, \lambda_U$, 
network architectures for $\{\mathcal{Y}_i\}_{i=0}^N$, $\{\mathcal{Z}_i\}_{i=0}^{N-1}$, $\{\mathcal{U}_i\}_{i=0}^{N-1}$.
\State \textbf{Initialize:} All network parameters $\theta$.
\For{\texttt{iter} = 1 to \texttt{max\_iter}}
    \State Generate $M$ independent Brownian increments $\{\Delta B_i^{(m)}\}_{i=0}^{N-1}$.
    \For{$m = 1, \dots, M$}
        \State $\hat X_0^{(m)} \gets a$
        \State $\hat Y_0^{(m)} \gets \mathcal{Y}_0(\hat X_0^{(m)};\theta_0^Y)$
        \State $\mathcal{L}_R^{(m)} \gets 0$, $\mathcal{L}_U^{(m)} \gets 0$
        \For{$i = 0, \dots, N-1$}
            \State $\hat Z_i^{(m)} \gets \mathcal{Z}_i(\hat X_i^{(m)};\theta_i^Z)$,
            $\hat U_i^{(m)} \gets \mathcal{U}_i(\hat X_i^{(m)};\theta_i^U)$
            \State $\hat X_{i+1}^{(m)} \gets \hat X_i^{(m)} + b(t_i, \hat X_i^{(m)}, \hat U_i^{(m)}, \hat Z_i^{(m)}) \Delta t + \sigma(t_i, \hat X_i^{(m)}, \hat U_i^{(m)}, \hat Z_i^{(m)}) \Delta B_i^{(m)}$
            \State $e_i^{(m)} \gets \mathcal{Y}_{i+1}(\hat X_{i+1}^{(m)};\theta_{i+1}^Y)$
            \State $\hat Y_{i+1}^{(m)} \gets \hat Y_i^{(m)} - f(t_i, \hat X_i^{(m)}, \hat Y_i^{(m)}, \hat Z_i^{(m)}) \Delta t + \;\hat Z_i^{(m)} \Delta B_i^{(m)} + e_i^{(m)}$
            \State $\mathcal{L}_R^{(m)} \gets \mathcal{L}_R^{(m)} + |e_i^{(m)}|^2 / \Delta t$
            \State $\mathcal{L}_U^{(m)} \gets \mathcal{L}_U^{(m)} + |\hat Y_i^{(m)} - \hat U_i^{(m)}|^2 \Delta t$
            \State $\hat X_i^{(m)} \gets \hat X_{i+1}^{(m)}$, $\hat Y_i^{(m)} \gets \hat Y_{i+1}^{(m)}$
        \EndFor
        \State $\mathcal{L}_T^{(m)} \gets |g(\hat X_N^{(m)}) - \hat Y_N^{(m)}|^2$
    \EndFor
    \State $\mathcal{L} \gets \frac{1}{M} \sum_{m=1}^M \big( \mathcal{L}_T^{(m)} + \lambda_R \mathcal{L}_R^{(m)} + \lambda_U \mathcal{L}_U^{(m)} \big)$
    \State Update all parameters $\theta$ using Adam optimizer on $\nabla_\theta \mathcal{L}$
\EndFor
\State \textbf{return} optimized parameters $\theta$.
\end{algorithmic}
\end{algorithm}

\subsection{Example 1: An FBSDE with Explicit Solution}
\label{ssec-example1}

We first consider a multi-dimensional fully coupled FBSDE with a linear-quadratic structure.
The exact solution is available through the associated Riccati equation,
which makes this example suitable for checking the relationship between the computable indicators and the true approximation error.  

The system is
\begin{equation}\label{eq:ex1}
    \begin{cases}
        dX_t^i=\bigl(-2X_t^i+Y_t^i\bigr)\,dt+\bigl(3X_t^i+Z_t^i\bigr)\,dB_t^i,\\[3pt]
        -dY_t^i=\bigl(-X_t^i-2Y_t^i+3Z_t^i\bigr)\,dt-Z_t^i\,dB_t^i,\\[3pt]
        X_0=a,\qquad Y_T=-QX_T,
    \end{cases}
\end{equation}
where
\[
    X_t=(X_t^1,\ldots,X_t^n)^\top,\qquad
    Y_t=(Y_t^1,\ldots,Y_t^n)^\top,\qquad
    Z_t=\operatorname{diag}(Z_t^1,\ldots,Z_t^n),
\]
and $B_t=(B_t^1,\ldots,B_t^n)^\top$ is an $n$-dimensional Brownian motion.

In this example, we set the dimension $n=100$ for the numerical experiments,
the initial condition $a=\mathbf{1}_n$ (the all-ones vector),
the terminal time $T=0.1$,
and the terminal coefficient $Q=5.0\,I_n$.
The time horizon is partitioned into $N$ equal subintervals.
Unless otherwise stated, we take $N=20$ as the default grid;
variations of $N$ (e.g., $N=10, 30, 40, 50$) are explicitly indicated when analyzing the discretization effect in Figure~\ref{fig:N_effect} and Table~\ref{tab:N_indicators}.

We seek a solution of the form
\[
    Y_t=-K_tX_t, \qquad Z_t=\operatorname{diag}(-M_tX_t).
\]
Substitution into \eqref{eq:ex1} gives the Riccati system
\begin{equation}\label{eq:Riccati}
    \begin{cases}
        \dot K_t-K_t^2-4K_t+3M_t+I_n=0,\\
        -3K_t+K_tM_t+M_t=0,\\
        K_T=Q.
    \end{cases}
\end{equation}
With the above parameters,
a Runge--Kutta solver for \eqref{eq:Riccati} gives
\[
    Y_0=-2.8730\,\mathbf 1_n .
\]
This value is used as the benchmark for the numerical approximation.


Figure~\ref{fig:training} reports the training dynamics for a representative run.
The terminal defect,
the pathwise residual,
and the control mismatch all decrease during training.
The error in the initial value $Y_0$ also decreases over the training iterations.
This indicates that,
for a fixed time grid and a fixed network configuration,
the computable loss is consistent with the improvement of the approximation.
We emphasize,
however,
that the absolute value of the loss should not be interpreted as a direct pointwise estimate of the error,
because it is affected by the time discretization,
the accumulation of local residuals,
the penalty weights,
and the optimization difficulty.

\begin{figure}[htb]
    \centering
    \begin{subfigure}[b]{0.48\textwidth}
        \centering
        \includegraphics[width=\textwidth]{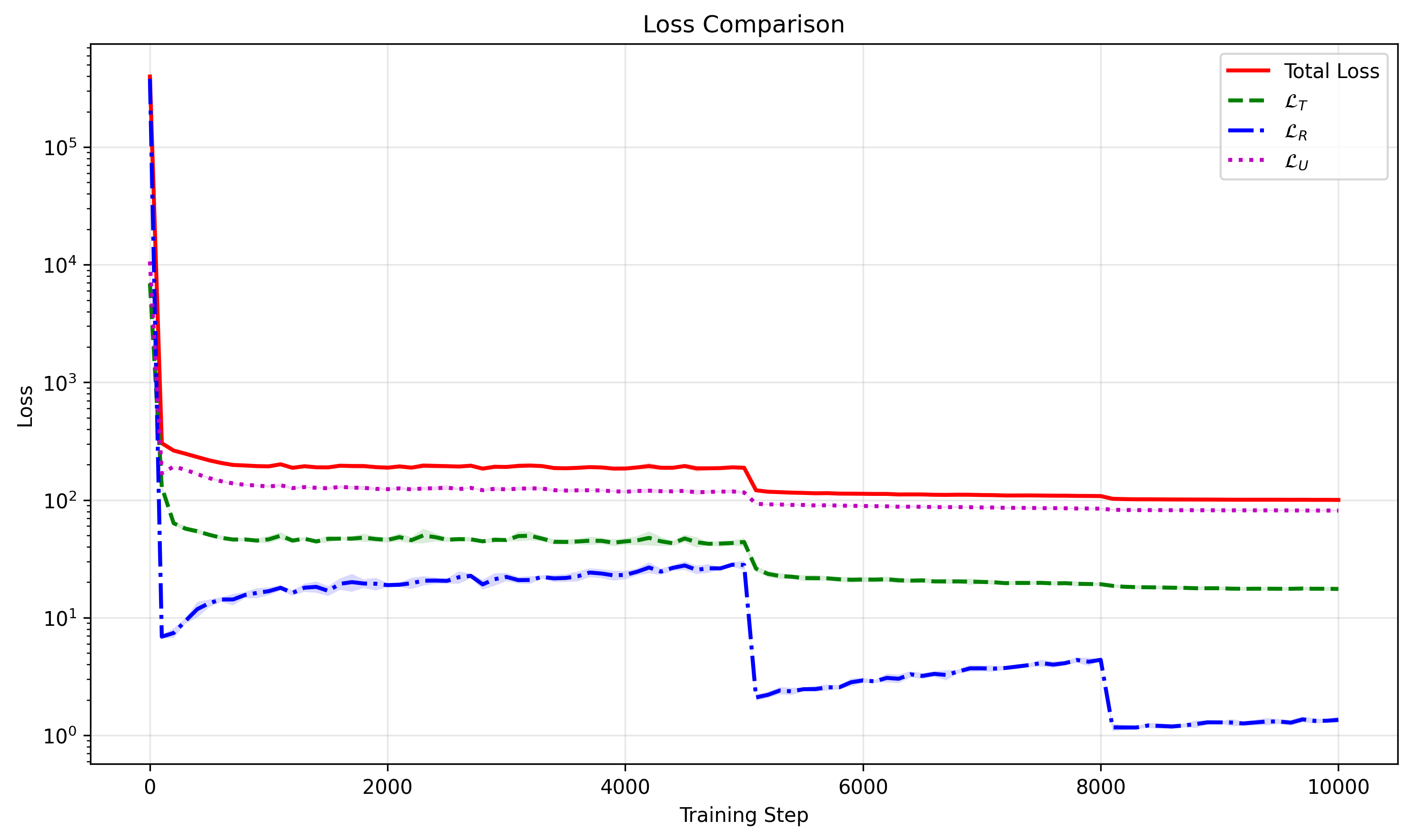}
        \caption{Training loss}
        \label{fig:sub-loss}
    \end{subfigure}
    \hfill
    \begin{subfigure}[b]{0.48\textwidth}
        \centering
        \includegraphics[width=\textwidth]{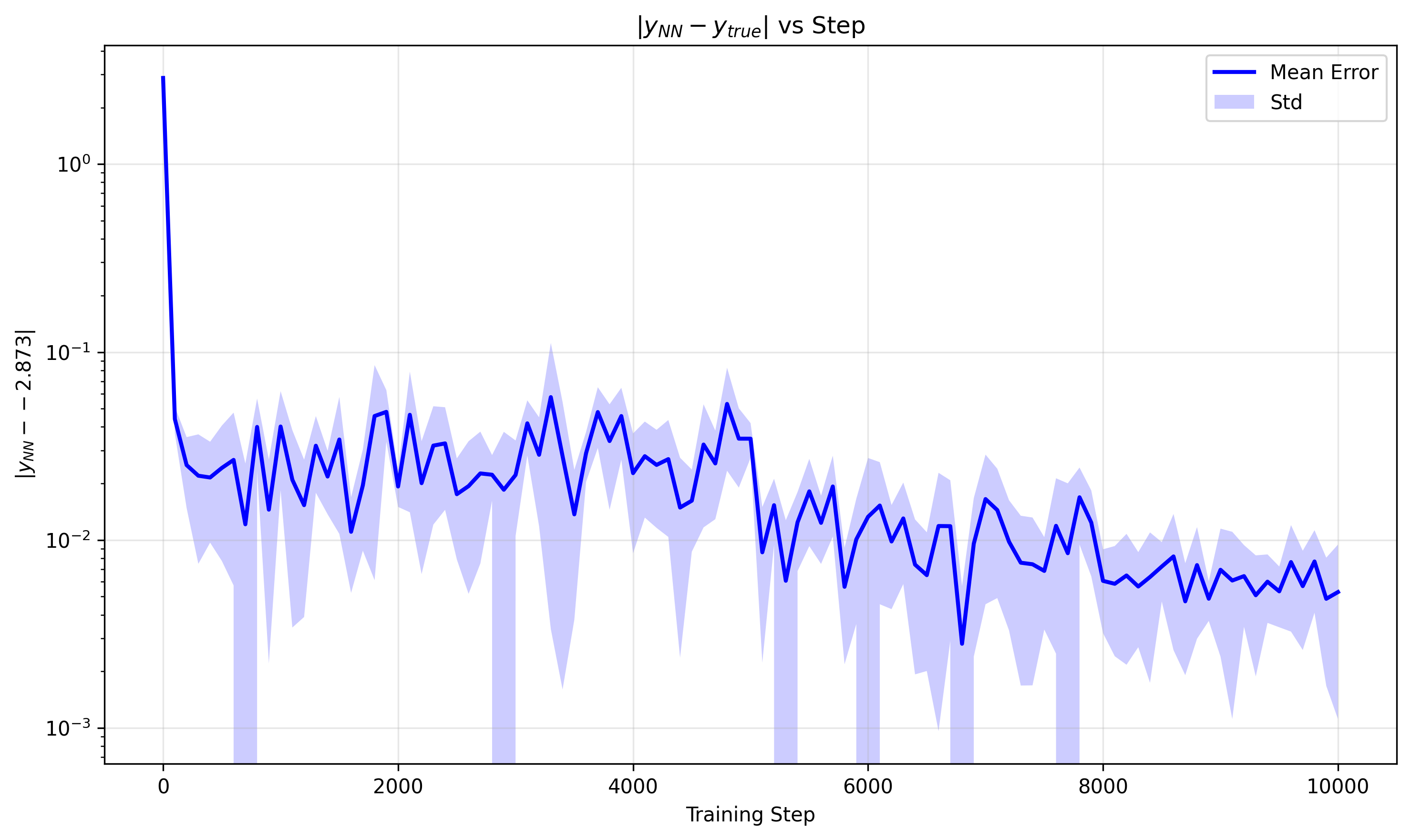}
        \caption{$Y_0$ error}
        \label{fig:sub-y0error}
    \end{subfigure}
    \caption{Training dynamics of the deep FBSDE solver.
        Panel (a) shows the total loss $\mathcal L$ and its components $\mathcal L_T$, $\mathcal L_R$,
        and $\mathcal L_U$ on a logarithmic scale.
        Panel (b) shows the absolute error $|Y_0^{\rm NN}-Y_0|$.
        The results suggest that the decrease of the computable a posteriori loss is accompanied by an improvement in the estimated initial value.}
    \label{fig:training}
\end{figure}

The influence of the time discretization is examined in Figure~\ref{fig:N_effect} and Table~\ref{tab:N_indicators}.
The table reports the terminal indicator,
the residual indicator,
the mismatch indicator,
the total loss,
and the absolute error in $Y_0$,
averaged over five independent runs.

\begin{figure}[htb]
    \centering
    \includegraphics[width=0.65\textwidth]{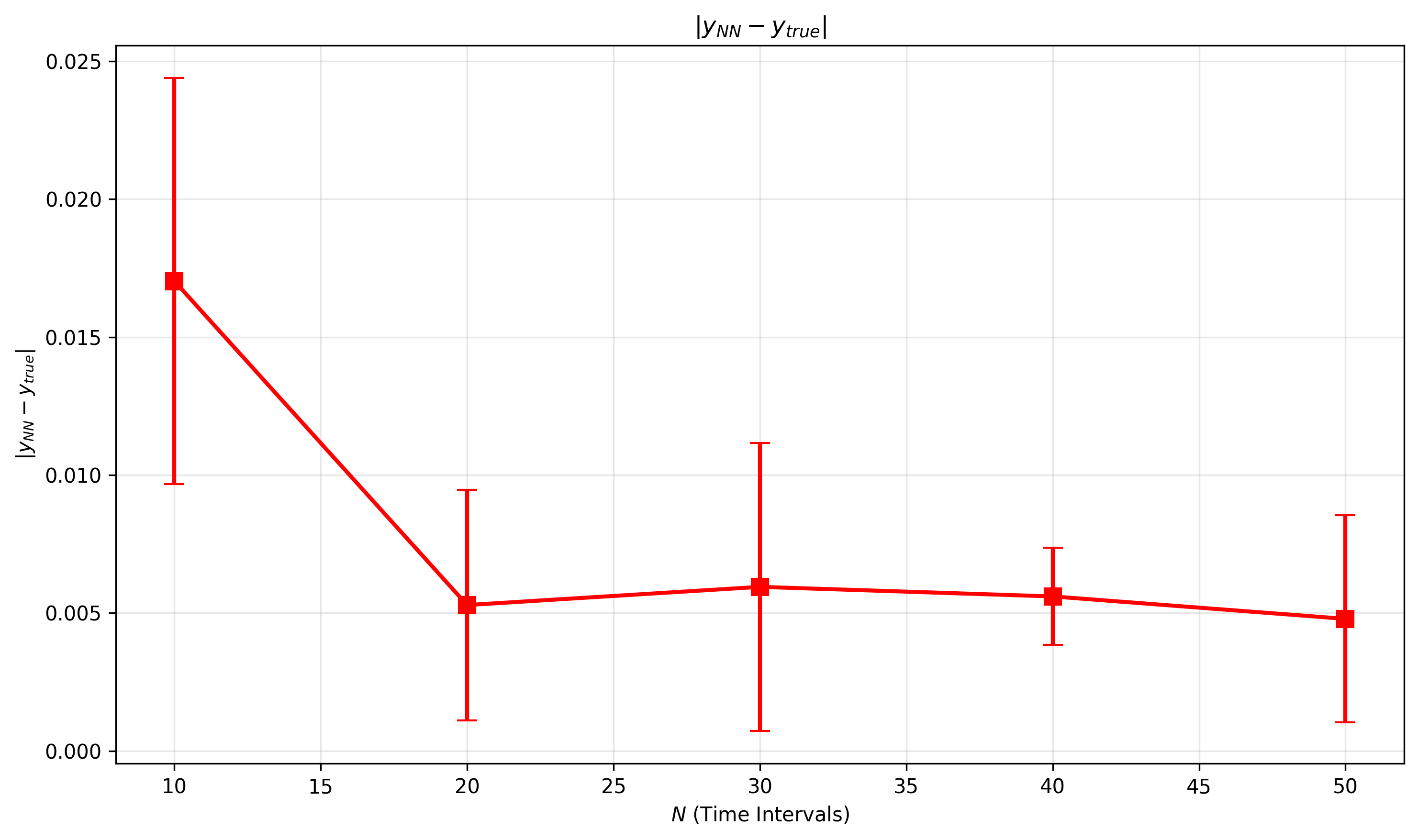}
    \caption{Absolute error $|Y_0^{\mathrm{NN}}-Y_0|$ for different numbers of time steps $N$,\
        averaged over five independent runs.
        The reference value is obtained from the Riccati equation.
        After $N=20$ the error stabilizes,
        and further grid refinement does not lead to a significant improvement in this experiment.}
    \label{fig:N_effect}
\end{figure}

\begin{table}[htb]
    \centering
    \caption{A posteriori indicators and $Y_0$ error for different $N$.
        Mean and standard deviation over five runs are reported.}
    \label{tab:N_indicators}
    \resizebox{\textwidth}{!}{
    \begin{tabular}{cccccc}
        \toprule
        $N$ & $\hat{\eta}^2$ & $\hat{\mathfrak{R}}_\pi$ & $\hat{\mathfrak{P}}_\pi$ & Total loss & $|Y_0^{\mathrm{NN}}-Y_0|$ \\
        \midrule
        10 & 8.191 $\pm$ 0.228 & 0.454 $\pm$ 0.036 & 33.75 $\pm$ 1.061 & 42.39 $\pm$ 1.191 & 0.017027 $\pm$ 0.0074 \\
        20 & 17.53 $\pm$ 0.611 & 1.354 $\pm$ 0.083 & 81.29 $\pm$ 1.400 & 100.2 $\pm$ 1.918 & 0.005288 $\pm$ 0.0042 \\
        30 & 31.28 $\pm$ 1.563 & 2.679 $\pm$ 0.040 & 128.6 $\pm$ 1.671 & 162.6 $\pm$ 3.090 & 0.005947 $\pm$ 0.0052 \\
        40 & 47.82 $\pm$ 2.407 & 3.909 $\pm$ 0.188 & 164.8 $\pm$ 1.464 & 216.5 $\pm$ 3.805 & 0.005601 $\pm$ 0.0018 \\
        50 & 62.52 $\pm$ 4.394 & 5.005 $\pm$ 0.263 & 197.1 $\pm$ 3.290 & 264.6 $\pm$ 7.137 & 0.004788 $\pm$ 0.0038 \\
        \bottomrule
    \end{tabular}
    }
\end{table}

Several observations can be made.
First, the total computable loss increases with $N$ in this experiment.
This does not contradict the a posteriori estimate,
since the empirical loss is a sum of local residual-type quantities and its magnitude is also influenced by the number of time steps,
the number of networks to be trained,
and the optimization landscape.
Second, the error in $Y_0$ drops markedly from $N=10$ to $N=20$ and then remains roughly stable,
fluctuating around $0.005$ without a clear monotone trend for larger $N$.
This indicates that the time discretization error is already well controlled with a moderate number of time steps,
while further grid refinement does not bring a significant additional benefit under the present network and training setup.
A finer grid reduces the local discretization error in principle,
but it also increases the number of neural network components and may make the optimization problem more difficult.
The results also confirm that the a posteriori indicators should be interpreted as diagnostic tools rather than as a simple monotone ranking across different discretizations.
For a fixed discretization,
their decrease reflects better satisfaction of the terminal condition,
the backward dynamics,
and the control consistency.
Across different discretizations,
they help reveal the interplay between discretization accuracy and optimization difficulty.

Figure~\ref{fig:trajectory} compares the true and approximated trajectories of $Y_t$ and $Z_t$ along one representative Brownian path for a 5-dim case with $N=50$.
The value process $Y$ is approximated very accurately along this path,
whereas the control process $Z$ exhibits a more visible discrepancy.
This is consistent with the fact that $Z$ is a gradient-type quantity and is usually more sensitive to local fluctuations of the Brownian path.
The result supports the use of residual and mismatch indicators as useful diagnostics,
but one should not infer full pathwise accuracy from a single trajectory alone.

\begin{figure}[htb]
    \centering
    \begin{subfigure}[b]{0.48\textwidth}
        \centering
        \includegraphics[width=\textwidth]{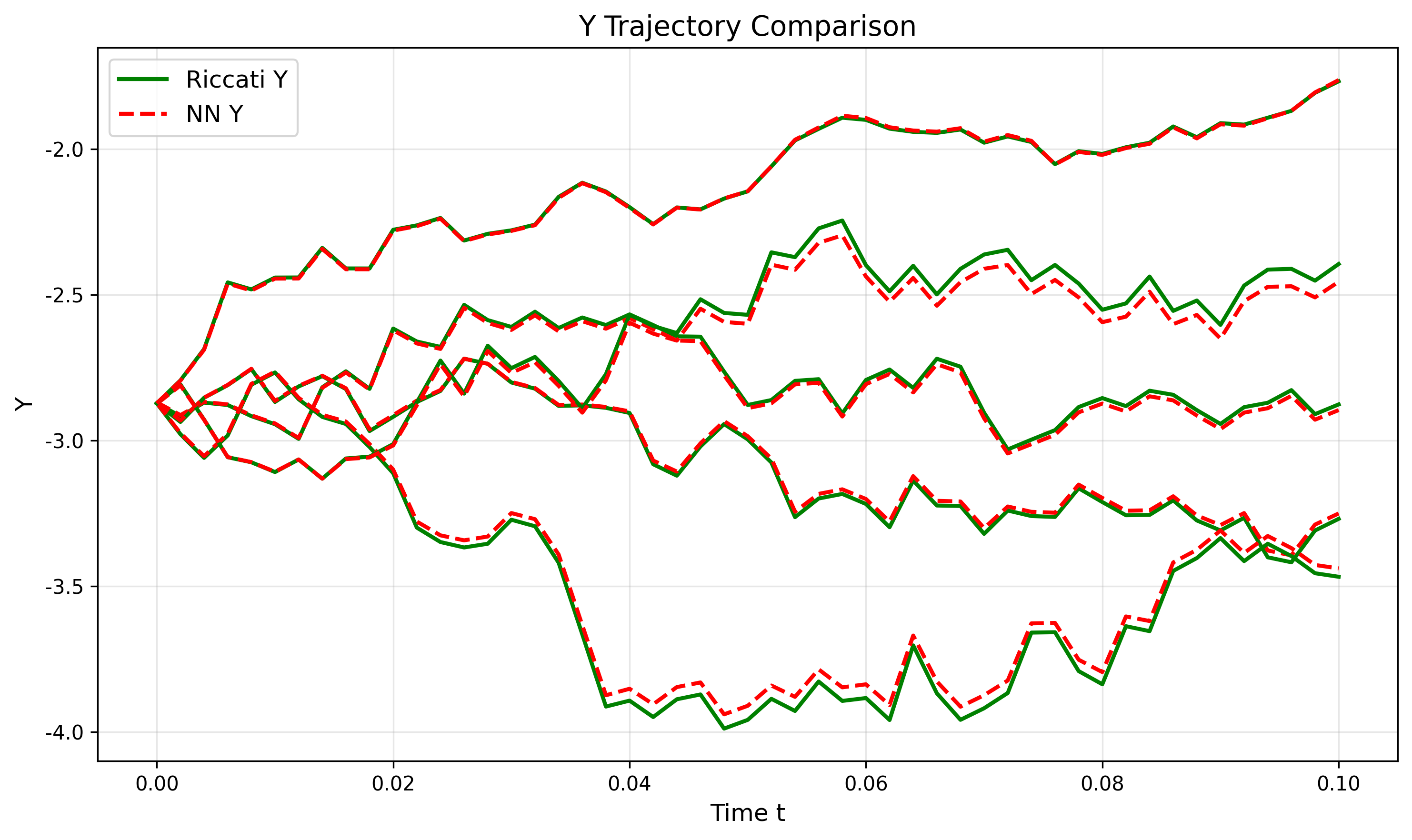}
        \caption{Trajectory of $Y_t$}
        \label{fig:sub-yfit}
    \end{subfigure}
    \hfill
    \begin{subfigure}[b]{0.48\textwidth}
        \centering
        \includegraphics[width=\textwidth]{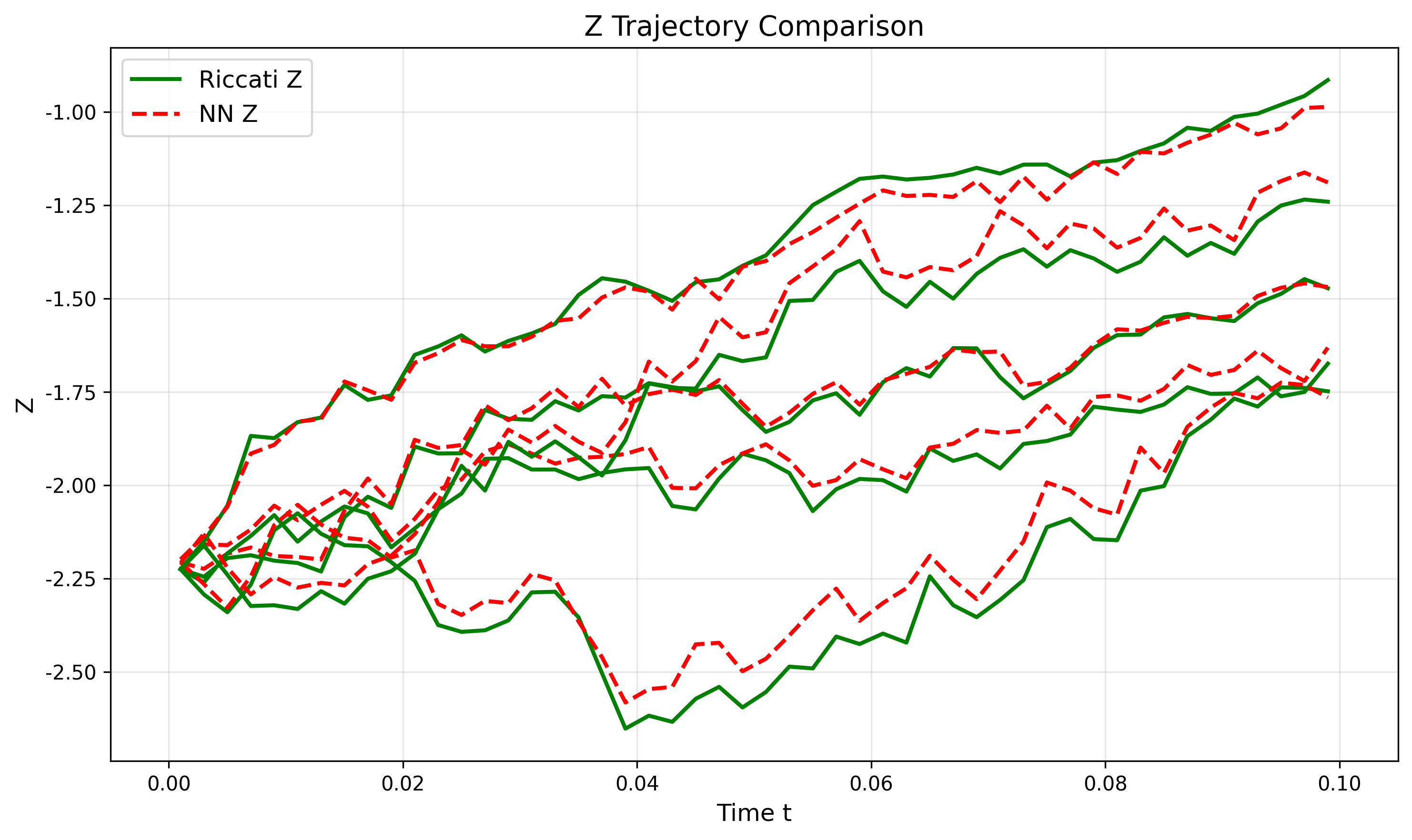}
        \caption{Trajectory of $Z_t$}
        \label{fig:sub-zfit}
    \end{subfigure}
    \caption{Comparison of the true trajectories and the neural network approximations for one representative path in the five-dimensional case with $N=50$.
        The $Y$ component is accurately reproduced along this path,
        while the $Z$ component is more difficult to approximate.}
    \label{fig:trajectory}
\end{figure}

\subsection{Example 2: FBSDE Associated with a Burgers-Type PDE}
\label{ssec-example2}

We next consider a Burgers-type fully coupled FBSDE with a tunable feedback strength.
Let $X_t\in\R^n$, $Y_t\in\R$, and $Z_t\in\R^n$.
For a viscosity parameter $\nu>0$,
a damping coefficient $\lambda\in\R$,
and a coupling strength $\rho\ge 0$,
consider
\begin{equation}\label{eq:ex2}
    \begin{cases}
        dX_t = -\rho Y_t\mathbf 1_n\,dt+\sqrt{2\nu}\,dB_t,\\[4pt]
        -dY_t = \lambda Y_t\,dt-Z_t^\top dB_t,\\[4pt]
        X_0=x_0,\qquad Y_T=g(X_T),
    \end{cases}
\end{equation}
where $B_t$ is an $n$-dimensional Brownian motion.
The parameter $\rho$ controls the strength of the feedback from the backward component $Y_t$ into the forward dynamics.
When $\rho=0$, the forward equation is decoupled from $Y$; when $\rho>0$,
the forward dynamics depend directly on the backward component.

Assume that the FBSDE admits a smooth decoupling field $u:[0,T]\times\R^n\to\R$ such that
\[
    Y_t=u(t,X_t),\qquad Z_t=\sqrt{2\nu}\,\nabla_x u(t,X_t).
\]
By It\^o's formula,
$u$ satisfies
\begin{equation}\label{eq-pde}
    \begin{cases}
        \partial_t u -\rho u\,\mathbf 1_n\cdot\nabla_xu + \nu\Delta_xu+\lambda u=0,
        \qquad (t,x)\in[0,T)\times\R^n,\\[4pt]
        u(T,x)=g(x).
    \end{cases}
\end{equation}
Equivalently, after the time reversal $v(\tau,x)=u(T-\tau,x)$,
one obtains the viscous damped Burgers-type equation~\cite{CruzeiroShamarova2011,Brze2014}
\[
    \partial_\tau v + \rho v\,\mathbf 1_n\cdot\nabla_xv = \nu\Delta_xv+\lambda v .
\]

In the numerical experiment,
we set
\[
    n=10,\qquad g(x)=\sin\!\left(\sum_{i=1}^{10}x_i\right),\qquad T=1,
\]
and
\[
    x_0=\mathbf 0,\qquad\nu=0.2,\qquad\lambda=0.5,\qquad\rho=0.5 .
\]

The time interval $[0,T]$ is uniformly partitioned into $N$ subintervals.
Unless otherwise stated, we take $N=20$ as the default grid.
Since no closed-form solution is available for this configuration,
this example is not intended to measure the true approximation error directly.
Instead, it examines whether the three computable components $\mathcal L_T$, $\mathcal L_R$, and $\mathcal L_U$
provide useful information about the consistency and stability of the trained solution.

We consider the following configurations:
\begin{itemize}
    \item \textbf{Full loss}: all three terms $\mathcal L_T$, $\mathcal L_R$, and $\mathcal L_U$ are included.
    The penalty weights $\lambda_R$ and $\lambda_U$ are varied to examine the sensitivity of the method.
    \item \textbf{w/o R}: the pathwise residual term $\mathcal L_R$ is removed,
    while the control mismatch term $\mathcal L_U$ is retained.
    \item \textbf{w/o RU}: both $\mathcal L_R$ and $\mathcal L_U$ are removed,
    leaving only the terminal loss.
\end{itemize}
The two ablation configurations are used to assess the numerical role of the residual and mismatch terms.
Since no reference solution is available,
the comparison should be interpreted in terms of stability, consistency, and sensitivity of the estimates,
rather than in terms of exact accuracy.

Figures~\ref{fig:loss_betaB} and~\ref{fig:loss_betaP} display the training loss curves for different choices of the penalty weights.
Each panel reports the total loss and its components on a logarithmic scale.
The loss curves show that the three components can be reduced simultaneously under the full-loss configuration.
This indicates that the trained networks can approximately satisfy the terminal condition, the discrete backward dynamics, and the control consistency condition at the same time.
The final loss levels vary with the penalty weights,
which is expected because $\lambda_R$ and $\lambda_U$ change the relative strength of the corresponding soft constraints.

\begin{figure}[htb!]
    \centering
    \includegraphics[width=1.0\textwidth]{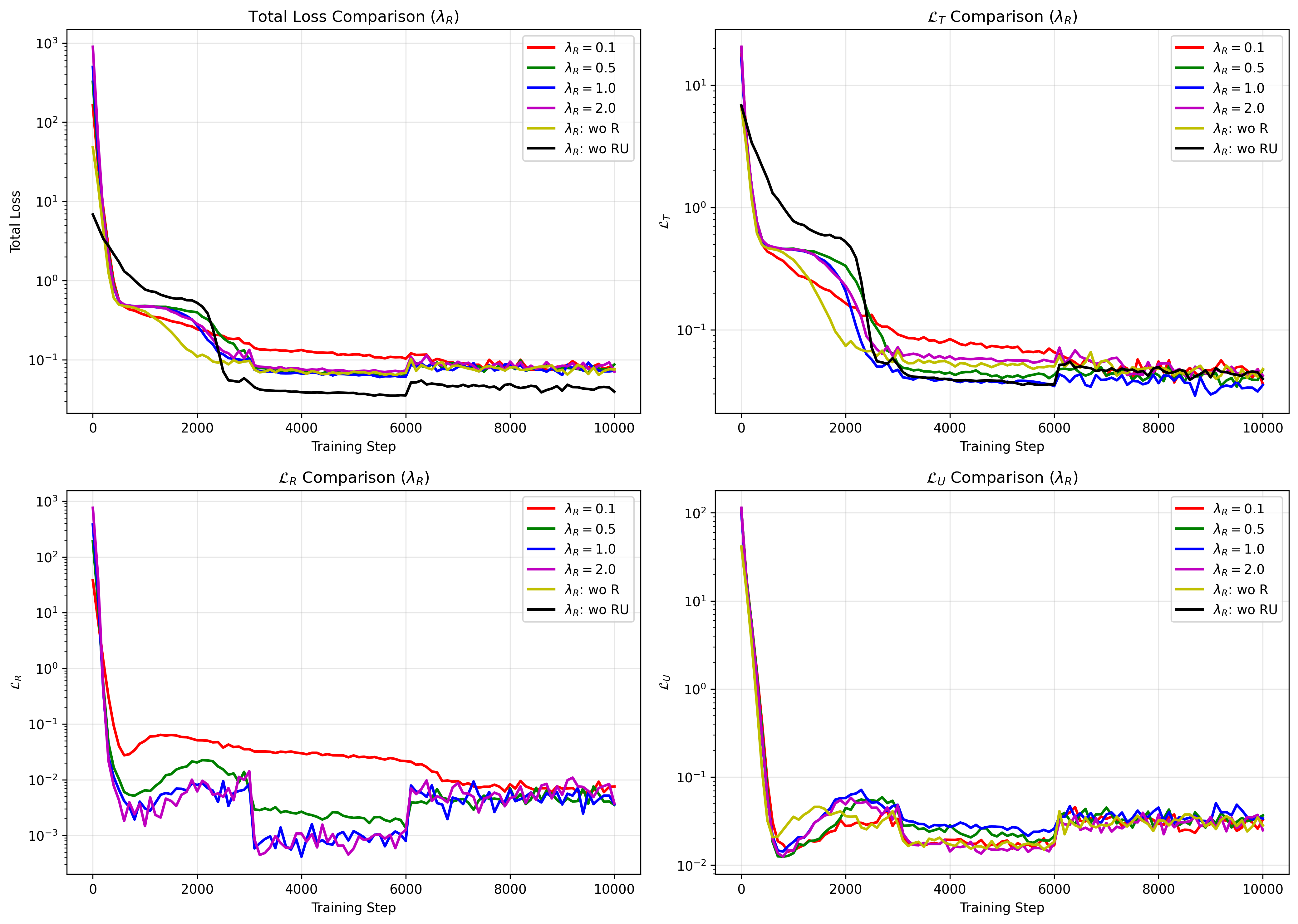}
    \caption{Training loss curves for different residual penalty weights $\lambda_R$.
    Each panel shows the total loss together with $\mathcal L_T$, $\mathcal L_R$, and $\mathcal L_U$.}
    \label{fig:loss_betaB}
\end{figure}

\begin{figure}[htb!]
    \centering
    \includegraphics[width=1.0\textwidth]{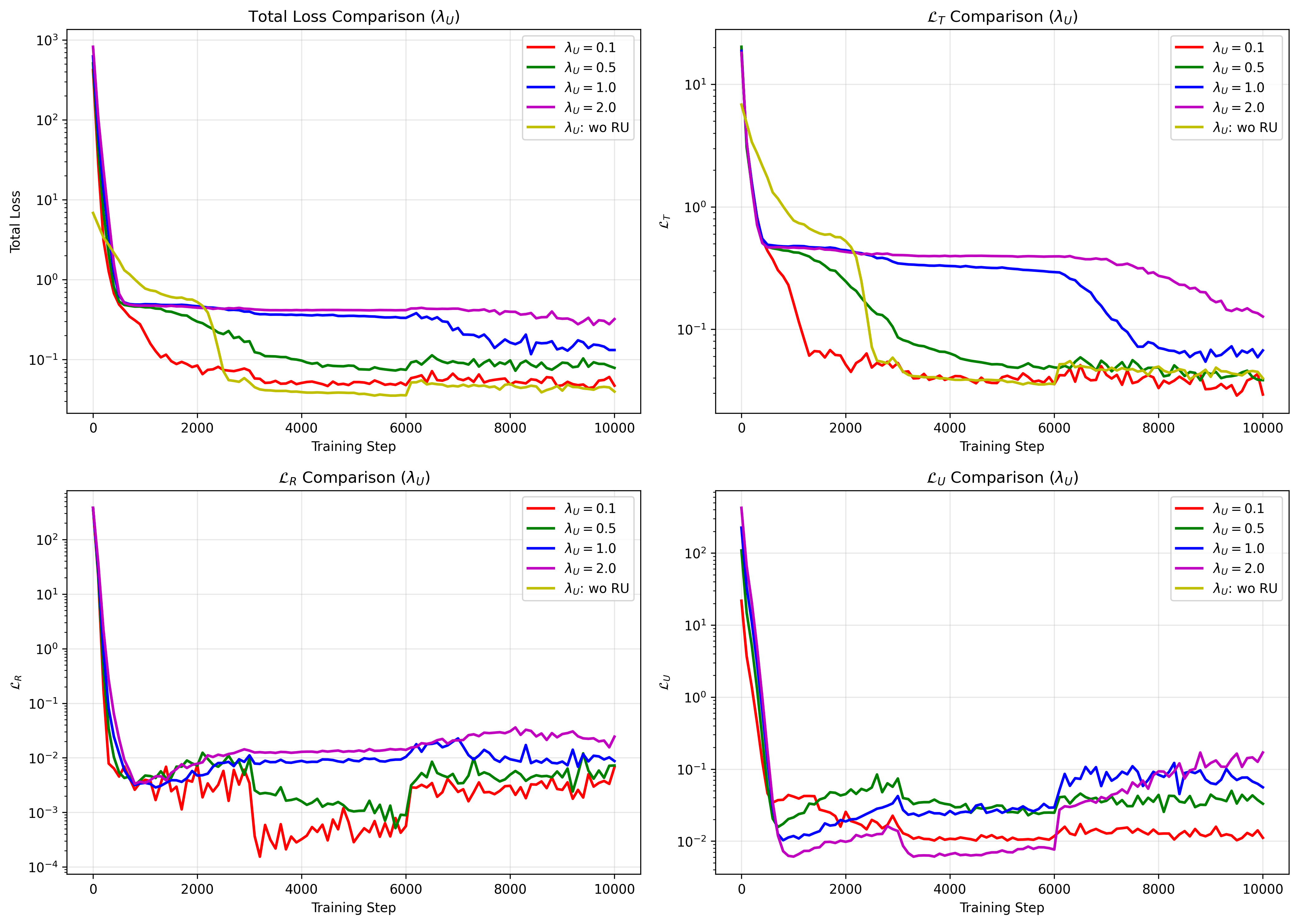}
    \caption{Training loss curves for different mismatch penalty weights $\lambda_U$.\
    The layout is the same as in Figure~\ref{fig:loss_betaB}.}
    \label{fig:loss_betaP}
\end{figure}

Tables~\ref{tab:ablation_betaB} and~\ref{tab:ablation_betaP} report the estimated values of $Y_0$ and a scalar summary of $Z_0$ for different configurations.
Since $Z_0\in\R^n$ in this example,
the reported value of $Z_0$ should be understood as the first value of its components in this example, namely$\bar Z_0:=Z_0^1$.

\begin{table}[htbp!]
    \centering
    \caption{Ablation study: estimated $Y_0$ and scalar summary $\bar Z_0$ for the full-loss configuration with different residual penalty weights $\lambda_R$, and for two ablation configurations.
    ``w/o R'' removes the pathwise residual $\mathcal L_R$; ``w/o RU'' removes both $\mathcal L_R$ and the control mismatch $\mathcal L_U$.
    Mean and standard deviation over five runs are reported.}
    \label{tab:ablation_betaB}
    \begin{tabular}{ccc}
        \toprule
        $\lambda_R$ / Configuration
        & $Y_0$ (mean $\pm$ std)
        & $\bar Z_0$ (mean $\pm$ std) \\
        \midrule
        0.1  & 0.256436 $\pm$ 0.003725 & 0.110142 $\pm$ 0.010282 \\
        0.5  & 0.242612 $\pm$ 0.003119 & 0.115508 $\pm$ 0.009206 \\
        1.0  & 0.237267 $\pm$ 0.006471 & 0.122293 $\pm$ 0.009611 \\
        2.0  & 0.229739 $\pm$ 0.003257 & 0.120487 $\pm$ 0.012230 \\
        w/o R  & 0.227853 $\pm$ 0.011740 & 0.135288 $\pm$ 0.009247 \\
        w/o RU & 0.233889 $\pm$ 0.002582 & 0.135016 $\pm$ 0.003326 \\
        \bottomrule
    \end{tabular}
\end{table}

\begin{table}[htb!]
    \centering
    \caption{Ablation study: estimated $Y_0$ and scalar summary $\bar Z_0$ for the full-loss configuration with different mismatch penalty weights $\lambda_U$, and for the baseline ``w/o RU''.  Mean and standard deviation over five runs are reported.}
    \label{tab:ablation_betaP}
    \begin{tabular}{ccc}
        \toprule
        $\lambda_U$ / Configuration
        & $Y_0$ (mean $\pm$ std)
        & $\bar Z_0$ (mean $\pm$ std) \\
        \midrule
        0.1  & 0.233716 $\pm$ 0.005595 & 0.138136 $\pm$ 0.006634 \\
        0.5  & 0.234612 $\pm$ 0.004430 & 0.119596 $\pm$ 0.012485 \\
        1.0  & 0.233662 $\pm$ 0.003395 & 0.120903 $\pm$ 0.004734 \\
        2.0  & 0.221752 $\pm$ 0.003079 & 0.086442 $\pm$ 0.018901 \\
        w/o RU & 0.233889 $\pm$ 0.002582 & 0.135016 $\pm$ 0.003326 \\
        \bottomrule
    \end{tabular}
\end{table}

\begin{figure}[htb!]
    \centering
    \begin{subfigure}[b]{0.48\textwidth}
        \centering
        \includegraphics[width=\textwidth]{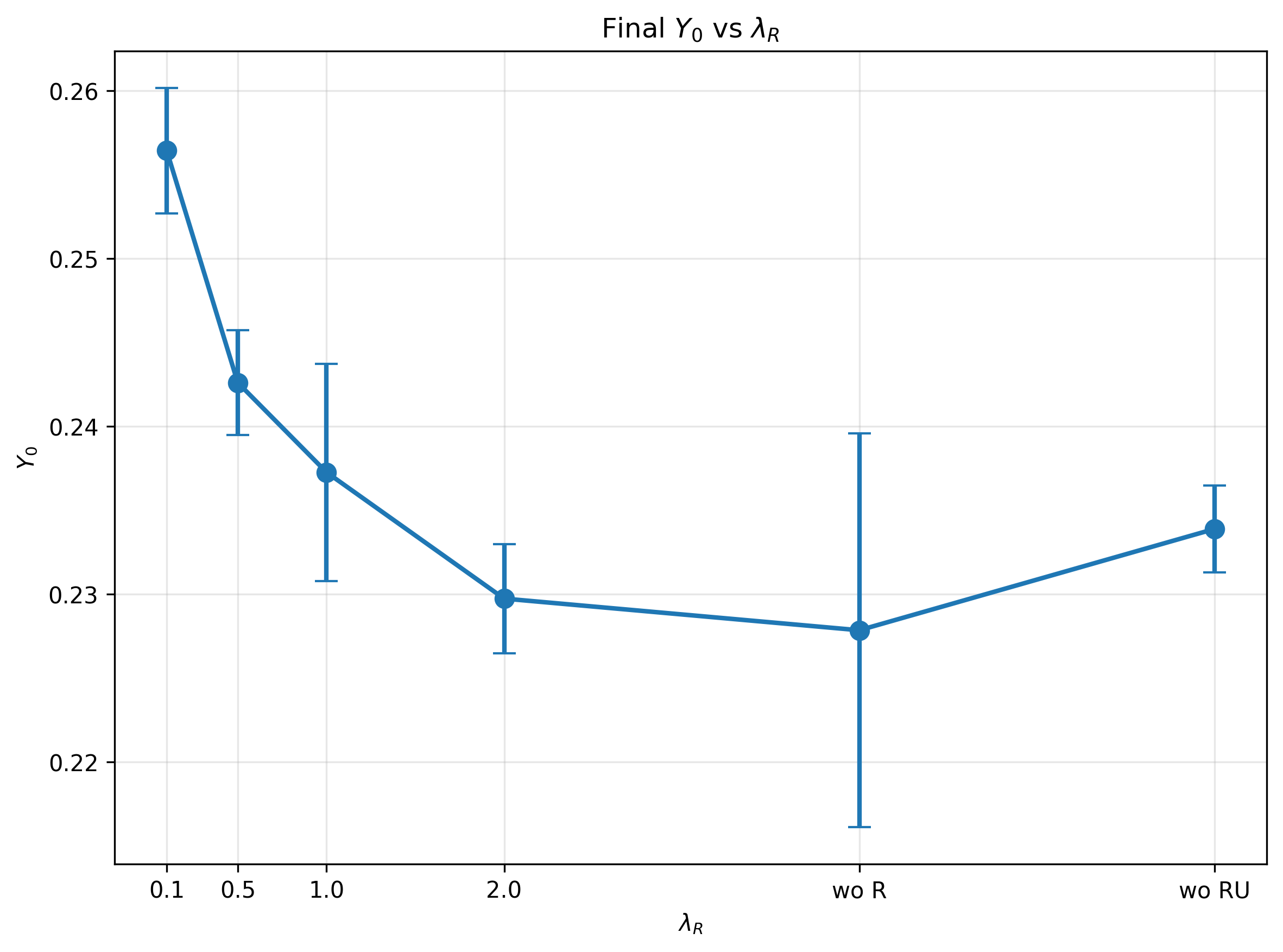}
        \caption{$Y_0$ vs.\ $\lambda_R$}
        \label{fig:sub-y0-betaB}
    \end{subfigure}
    \hfill
    \begin{subfigure}[b]{0.48\textwidth}
        \centering
        \includegraphics[width=\textwidth]{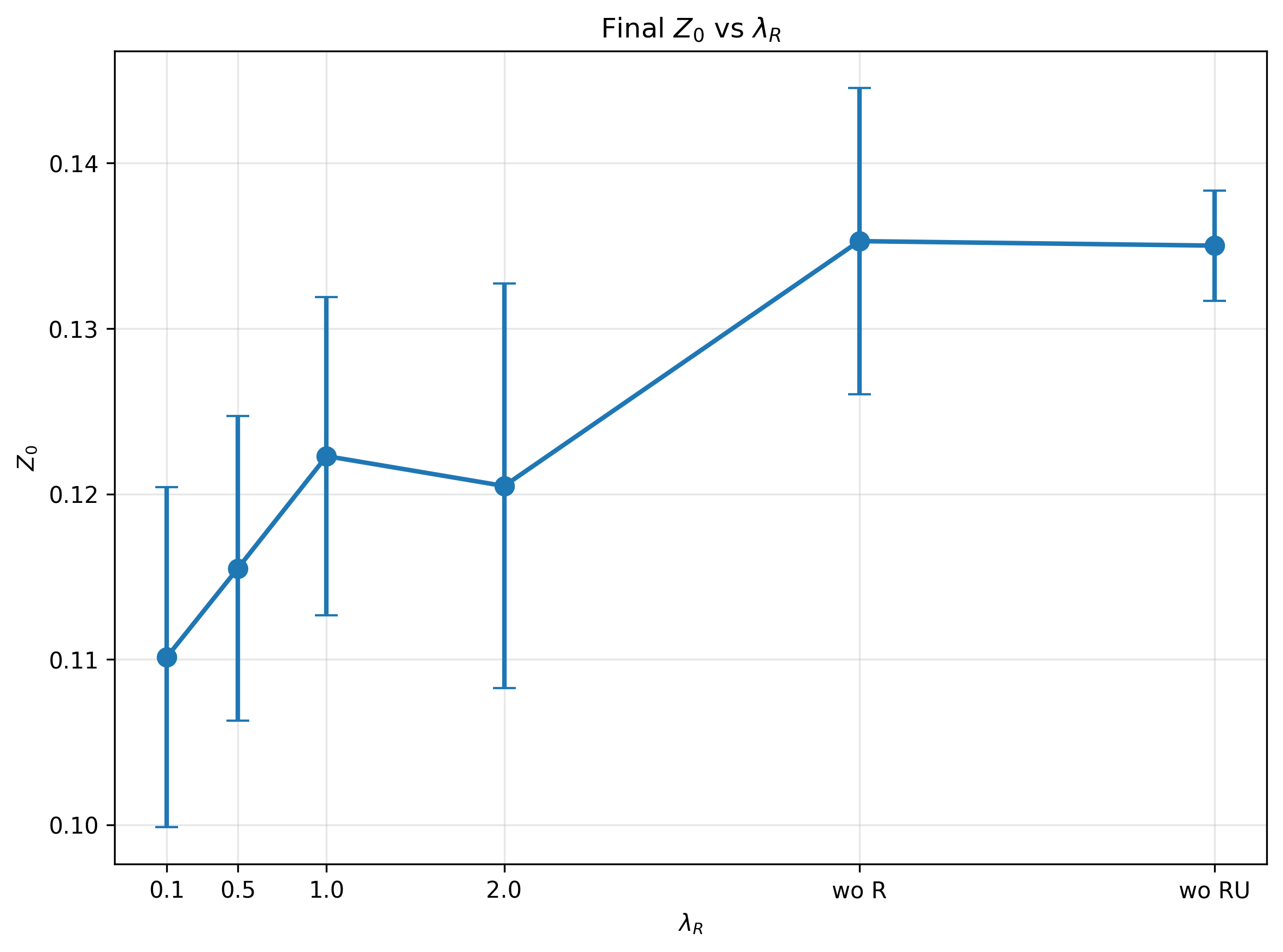}
        \caption{$\bar Z_0$ vs.\ $\lambda_R$}
        \label{fig:sub-z0-betaB}
    \end{subfigure}
    \caption{Dependence of the estimated $Y_0$ and $\bar Z_0$ on the residual penalty weight $\lambda_R$.}
    \label{fig:ablation_yz_betaB}
\end{figure}

\begin{figure}[htbp!]
    \centering
    \begin{subfigure}[b]{0.48\textwidth}
        \centering
        \includegraphics[width=\textwidth]{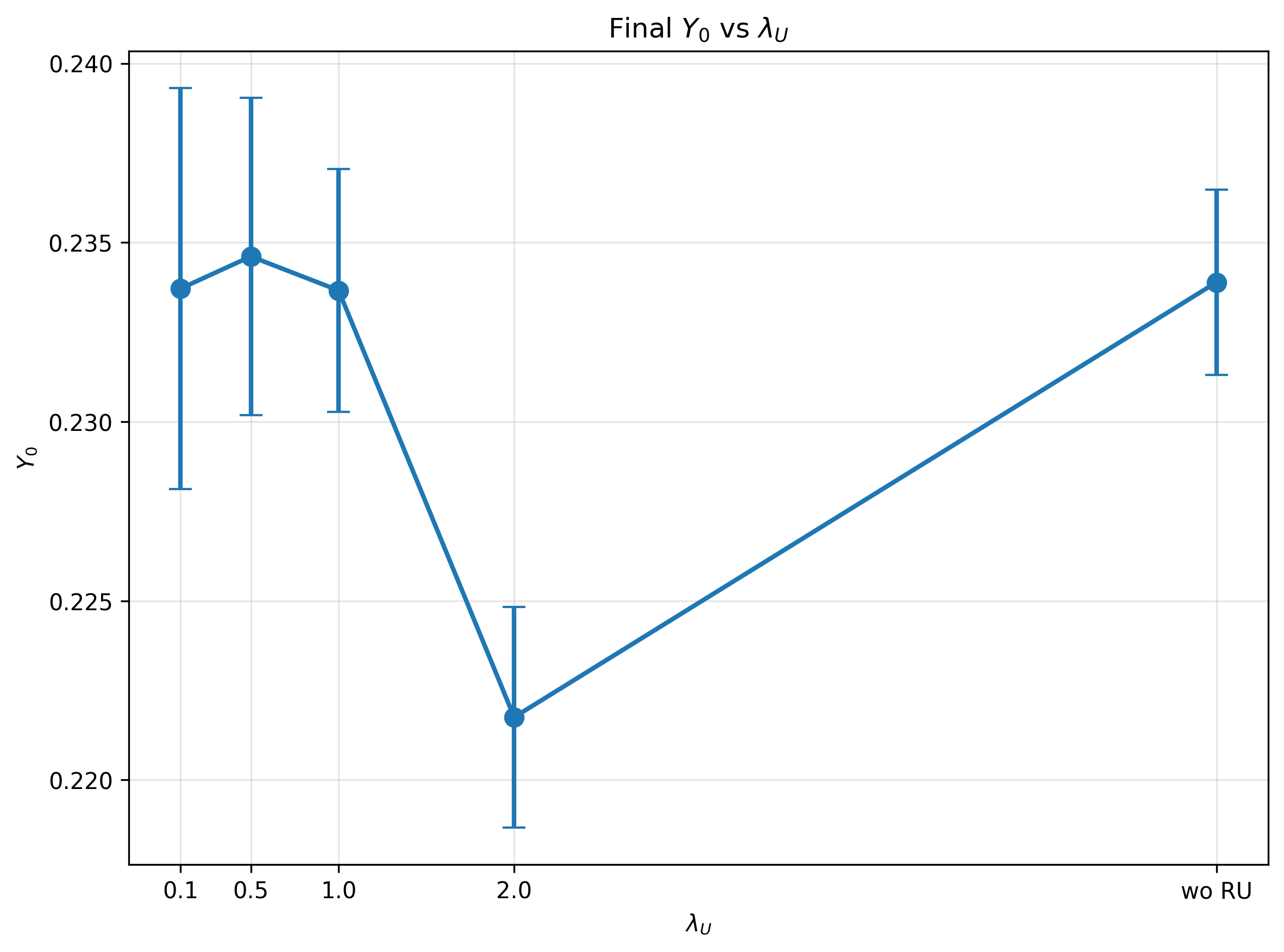}
        \caption{$Y_0$ vs.\ $\lambda_U$}
        \label{fig:sub-y0-betaP}
    \end{subfigure}
    \hfill
    \begin{subfigure}[b]{0.48\textwidth}
        \centering
        \includegraphics[width=\textwidth]{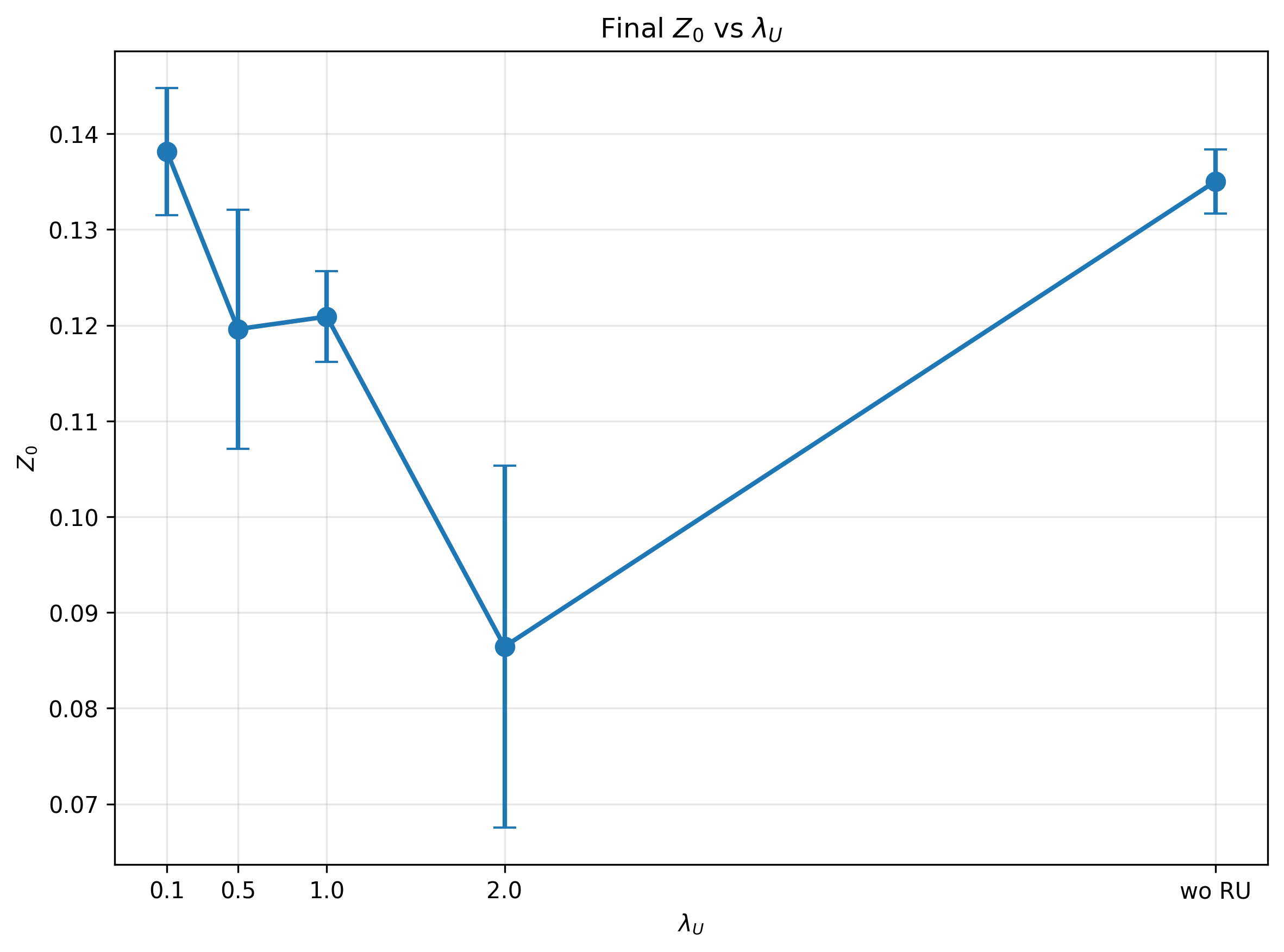}
        \caption{$\bar Z_0$ vs.\ $\lambda_U$}
        \label{fig:sub-z0-betaP}
    \end{subfigure}
    \caption{Dependence of the estimated $Y_0$ and $\bar Z_0$ on the mismatch penalty weight $\lambda_U$.}
    \label{fig:ablation_yz_betaP}
\end{figure}

The ablation results lead to the following observations.
First, the full-loss configuration produces stable estimates over a range of penalty weights.
The estimated $Y_0$ and the scalar summary of $Z_0$ change with $\lambda_R$ and $\lambda_U$,
but the variation is moderate for most parameter choices.
This suggests that the method is not overly sensitive to a single choice of penalty weight,
although excessively large or small weights may still affect the balance between the different constraints.
Second, removing the pathwise residual or the mismatch penalty changes the estimated control component more visibly than the estimated value component.
In particular, the ablation configurations tend to produce larger values of the reported $Z_0$ summary.
Since no exact solution is available in this example, this should not be described as a direct loss of accuracy.
A more appropriate interpretation is that the residual and mismatch penalties influence the consistency of the learned control process with the discrete FBSDE dynamics.
Third, the configuration without the residual term shows a larger variability in the estimate of $Y_0$ in Table~\ref{tab:ablation_betaB}.
This supports the view that the residual term plays a stabilizing role in the training process.
The terminal loss alone can enforce the final condition, but it does not directly control the intermediate backward dynamics.
The residual term therefore acts as an additional soft constraint that helps prevent the learned trajectory from fitting the terminal condition while deviating from the FBSDE structure in the interior of the time interval.

Overall, the Burgers-type experiment provides numerical evidence that the three-term loss suggested by the a posteriori estimate is useful in practice.
In the absence of a reference solution,
the indicators cannot prove the exact accuracy of the approximation,
but they provide computable diagnostics for terminal consistency, dynamic consistency, and control consistency.
This is precisely the practical role expected from the a posteriori framework.

\section{Conclusion}
\label{sec-conclusion}

In this paper, we developed an a posteriori error estimation framework for deep learning approximations of fully coupled forward-backward stochastic differential equations.
The analysis allows the forward coefficients to depend on an auxiliary control process that may differ from the backward component,
a formulation motivated by the decoupled training scheme in the second algorithm of~\cite{Peng_FBSDE_numerical} and by practical stabilization techniques such as target networks.
By introducing an explicit pathwise residual constraint $\mathfrak R_\pi$ into this decoupled framework, we obtained a fully computable a posteriori error bound that covers the terminal defect,
the pathwise residual, and the control mismatch.

The main theoretical contribution is a stability estimate for fully coupled FBSDEs under simultaneous perturbations of the drift, diffusion, generator, terminal condition, and control input.
Based on this continuous-time stability result, we derived a discrete-time a posteriori error estimate for arbitrary adapted approximations.
The resulting bound depends on three computable quantities: the terminal defect, the pathwise residual, and the control mismatch.
These quantities can be evaluated from the numerical solution itself and do not require prior knowledge of the exact solution.
In this sense, the estimate provides a genuine a posteriori diagnostic for deep FBSDE solvers.

The numerical experiments support the relevance of the theoretical framework.
For the linear-quadratic example with an explicit Riccati solution,
the computable indicators decrease during training and are consistent with the improvement of the estimated initial value.
The experiments with different time steps also show that a finer grid does not automatically lead to a better neural approximation,
because the increased number of time steps may make the optimization problem more difficult.
For the Burgers-type example without a closed-form reference solution,
the ablation studies indicate that the pathwise residual and the control mismatch terms play important stabilizing roles.
Although the absence of an exact solution prevents a direct measurement of the true error,
the indicators provide useful information about terminal consistency, dynamic consistency, and control consistency.

Several questions remain open.
The present analysis does not quantify how the choice of neural network architecture,
the number of trainable parameters,
the optimization algorithm,
or the Monte Carlo sample size affects the constants in the error estimates.
A systematic investigation of these issues would help connect the present a posteriori theory with more explicit convergence rate results.
It would also be interesting to extend the framework to FBSDEs with mean-field interaction,
reflecting boundaries, or more general path-dependent coefficients.
These topics are left for future research.

\bibliographystyle{alpha}
\bibliography{ref}

@article{hu1995solution,
title={Solution of forward-backward stochastic differential equations},
author={Hu, Ying and Peng, Shige},
journal={Probability Theory and Related Fields},
volume={103},
number={2},
pages={273--283},
year={1995}}

@article{PardouxPeng1990,
  author    = {Pardoux, \'Etienne and Peng, Shige},
  title     = {Adapted solution of a backward stochastic differential equation},
  journal   = {Systems and Control Letters},
  volume    = {14},
  number    = {1},
  pages     = {55--61},
  year      = {1990}
}

@article{PengWu1999,
  author    = {Peng, Shige and Wu, Zhen},
  title     = {Fully coupled forward-backward stochastic differential equations and applications to optimal control},
  journal   = {SIAM Journal on Control and Optimization},
  volume    = {37},
  number    = {3},
  pages     = {825--843},
  year      = {1999}
}

@article{ma1994solving,
  title={Solving forward-backward stochastic differential equations explicitly—a four step scheme},
  author={Ma, Jin and Protter, Philip and Yong, Jiongmin},
  journal={Probability theory and related fields},
  volume={98},
  number={3},
  pages={339--359},
  year={1994},
  publisher={Springer}
}

@book{MaYong1999,
  author    = {Ma, Jin and Yong, Jiongmin},
  title     = {Forward-Backward Stochastic Differential Equations and Their Applications},
  series    = {Lecture Notes in Mathematics},
  volume    = {1702},
  publisher = {Springer},
  year      = {1999}
}

@book{Bellman1957,
  author    = {Richard E. Bellman},
  title     = {Dynamic Programming},
  publisher = {Princeton University Press},
  year      = {1957}
}

@article{BouchardEkeland2004,
  author    = {Bouchard, Bruno and Ekeland, Ivar and Touzi, Nizar},
  title     = {On the {Malliavin} approach to {Monte Carlo} approximation of conditional expectations},
  journal   = {Finance and Stochastics},
  volume    = {8},
  number    = {1},
  pages     = {45--71},
  year      = {2004},
  doi       = {10.1007/s00780-003-0109-0}
}

@incollection{CruzeiroShamarova2011,
  author    = {Cruzeiro, Ana Bela and Shamarova, Evelina},
  title     = {On a forward-backward stochastic system associated to the {B}urgers equation},
  booktitle = {Stochastic Analysis with Financial Applications},
  series    = {Progress in Probability},
  volume    = {65},
  pages     = {43--59},
  year      = {2011},
  publisher = {Springer},
  address   = {New York, NY},
  doi       = {10.1007/978-3-0348-0097-6_4}
}

@article{Brze2014,
author = {Brze\'{z}niak, Zdzis\l{}aw and Goldys, Ben and Neklyudov, Misha},
title = {Multidimensional Stochastic Burgers Equation},
journal = {SIAM Journal on Mathematical Analysis},
volume = {46},
number = {1},
pages = {871-889},
year = {2014},
doi = {10.1137/120866117},

URL = { 
        https://doi.org/10.1137/120866117
},
eprint = { 
        https://doi.org/10.1137/120866117
}
}

@article{ReisingerStockingerZhang2024,
  author  = {Reisinger, Christoph and Stockinger, Wolfgang and Zhang, Yufei},
  title   = {A posteriori error estimates for fully coupled {McKean}--{Vlasov} forward-backward {SDEs}},
  journal = {IMA Journal of Numerical Analysis},
  volume  = {44},
  number  = {4},
  pages   = {2323--2369},
  year    = {2024},
  doi     = {10.1093/imanum/drad060}
}

@article{GermainMikaelWarin2022,
  author  = {Germain, Maximilien and Mikael, Joseph and Warin, Xavier},
  title   = {Numerical Resolution of {McKean}-{Vlasov} {FBSDEs} Using Neural Networks},
  journal = {Methodology and Computing in Applied Probability},
  volume  = {24},
  number  = {4},
  pages   = {2557--2586},
  year    = {2022},
  doi     = {10.1007/s11009-022-09946-z}
}

@article{BouchardTouzi2004,
  author    = {Bouchard, Bruno and Touzi, Nizar},
  title     = {Discrete-time approximation and {Monte-Carlo} simulation of backward stochastic differential equations},
  journal   = {Stochastic Processes and their Applications},
  volume    = {111},
  number    = {2},
  pages     = {175--206},
  year      = {2004}
}

@article{GobetLabart2008,
  author    = {Gobet, Emmanuel and Labart, C\'eline},
  title     = {Error expansion for the discretization of backward stochastic differential equations},
  journal   = {Stochastic Processes and their Applications},
  volume    = {118},
  number    = {5},
  pages     = {803--829},
  year      = {2008}
}

@article{HanJentzenE2018,
  author    = {Han, Jiequn and Jentzen, Arnulf and E, Weinan},
  title     = {Solving high-dimensional partial differential equations using deep learning},
  journal   = {Proceedings of the National Academy of Sciences},
  volume    = {115},
  number    = {34},
  pages     = {8505--8510},
  year      = {2018}
}

@article{RaissiPerdikarisKarniadakis2019,
  author    = {Raissi, Maziar and Perdikaris, Paris and Karniadakis, George E.},
  title     = {Physics-informed neural networks: A deep learning framework for solving forward and inverse problems involving nonlinear partial differential equations},
  journal   = {Journal of Computational Physics},
  volume    = {378},
  pages     = {686--707},
  year      = {2019}
}

@article{WeinanHanJentzen2017,
  author    = {E, Weinan and Han, Jiequn and Jentzen, Arnulf},
  title     = {Deep learning-based numerical methods for high-dimensional parabolic partial differential equations and backward stochastic differential equations},
  journal   = {Communications in Mathematics and Statistics},
  volume    = {5},
  number    = {4},
  pages     = {349--380},
  year      = {2017}
}

@article{HanLong2020,
  author    = {Han, Jiequn and Long, Jihao},
  title     = {Convergence of the deep {BSDE} method for coupled {FBSDEs}},
  journal   = {Probability, Uncertainty and Quantitative Risk},
  volume    = {5},
  number    = {1},
  pages     = {1--33},
  year      = {2020}
}

@article{GnoattoOberprillerPicarelli2025,
  author    = {Gnoatto, Alessandro and Oberpriller, Katharina and Picarelli, Athena},
  title     = {Convergence of a deep {BSDE} solver with jumps},
  journal   = {arXiv:2501.09727},
  year      = {2025}
}

@article{GaoEtAl2023,
  author    = {Gao, Chengfan and Chen, Sheng and Zhu, Zimu and Wang, Zhaojun},
  title     = {Convergence of the backward deep {BSDE} method with applications to optimal stopping problems},
  journal   = {SIAM Journal on Financial Mathematics},
  year      = {2023}
}

@article{Zhao2016Multistep,
  title={Multistep Schemes for Forward Backward Stochastic Differential Equations with Jumps},
  author={Fu, Yu and Zhao, Weidong and Zhou, Tao},
  journal={Journal of Scientific Computing},
  volume={69},
  number={2},
  pages={1-22},
  year={2016},
}

@InProceedings{Bender2012Monte,
author="Bender, Christian
and Steiner, Jessica",
editor="Carmona, Ren{\'e}  A.
and Del Moral, Pierre
and Hu, Peng
and Oudjane, Nadia",
title="Least-Squares Monte Carlo for Backward SDEs",
booktitle="Numerical Methods in Finance",
year="2012",
publisher="Springer Berlin Heidelberg",
address="Berlin, Heidelberg",
pages="257--289",
isbn="978-3-642-25746-9"
}

@article{AnderssonAnderssonOosterlee2025,
  author    = {Andersson, Kristoffer and Andersson, Adam and Oosterlee, Cornelis W.},
  title     = {The deep multi-{FBSDE} method: a robust deep learning method for coupled {FBSDEs}},
  journal   = {arXiv:2503.13193},
  year      = {2025}
}

@article{wu1999Fully,
title={Fully Coupled Forward-Backward Stochastic Differential Equations and Applications to Optimal Control},
author={Peng, Shige and Wu, Zhen},
journal={{SIAM} Journal on Control and Optimization},
volume={37},
number={3},
pages={825-843},
year={1999},
}

@article{peng1990stochasticmax,
title={A general stochastic maximum principle for optimal control problems},
author={Peng, Shige},
journal={Siam Journal on Control and Optimization},
volume={28},
number={4},
pages={966--979},
year={1990}}

@article{BenderZhang2008,
title={Time Discretization and Markovian Iteration for Coupled FBSDEs},
author={Christian Bender and Zhang, Jianfeng},
journal={Annals of Applied Probability},
volume = {18},
number = {1},
pages = {143-177},
year={2008}}

@ARTICLE{Peng_FBSDE_numerical,
author={Ji, Shaolin and Peng, Shige and Peng, Ying and Zhang, Xichuan},
journal={IEEE Intelligent Systems}, 
title={Three Algorithms for Solving High-Dimensional Fully Coupled FBSDEs Through Deep Learning}, 
year={2020},
volume={35},
number={3},
pages={71-84},}

@article{pham2018deep_1,
author = {Hur\'{e}, C\^{o}me and Pham, Huy\^{e}n and Bachouch, Achref and Langren\'{e}, Nicolas},
title = {Deep Neural Networks Algorithms for Stochastic Control Problems on Finite Horizon: Convergence Analysis},
journal = {SIAM Journal on Numerical Analysis},
volume = {59},
number = {1},
pages = {525-557},
year = {2021},
doi = {10.1137/20M1316640},

URL = { 
    
        https://doi.org/10.1137/20M1316640
    
    

},
eprint = { 
    
        https://doi.org/10.1137/20M1316640
    
}
}

@article{GobetTurkedjiev2016,
  author    = {Gobet, Emmanuel and Turkedjiev, Plamen},
  title     = {Linear regression {MDP} scheme for discrete backward stochastic differential equations under general conditions},
  journal   = {Mathematics of Computation},
  volume    = {85},
  number    = {299},
  pages     = {1359--1391},
  year      = {2016},
  doi       = {10.1090/mcom/3013}
}

@article{GobetLemorWarin2005,
  author    = {Gobet, Emmanuel and Lemor, Jean-Philippe and Warin, Xavier},
  title     = {A regression-based Monte Carlo method to solve backward stochastic differential equations},
  journal   = {Annals of Applied Probability},
  volume    = {15},
  number    = {3},
  pages     = {2172--2202},
  year      = {2005}
}

@article{pham2018deep_2,
title={Deep Neural Networks Algorithms for Stochastic Control Problems on Finite Horizon: Numerical Applications},
author={Bachouch, Achref and Hur\'{e}, C\^{o}me and Langren\'{e}, Nicolas and Pham, Huy\^{e}n},
journal={Methodology and Computing in Applied Probability},
year={2019},
volume={24},
number={1},
pages={143-178}
}

@article{finance1997,
author = {El Karoui, Nicole and Peng, Shige and Quenez, Marie Claire},
title = {Backward Stochastic Differential Equations in Finance},
journal = {Mathematical Finance},
volume = {7},
number = {1},
pages = {1-71},
doi = {https://doi.org/10.1111/1467-9965.00022},
url = {https://onlinelibrary.wiley.com/doi/abs/10.1111/1467-9965.00022},
eprint = {https://onlinelibrary.wiley.com/doi/pdf/10.1111/1467-9965.00022},
year = {1997}
}

@incollection{ElKarouiQuenez1997,
  author    = {El Karoui, Nicole and Quenez, Marie Claire},
  title     = {Imperfect Markets and Backward Stochastic Differential Equations},
  booktitle = {Numerical Methods in Finance},
  editor    = {Rogers, L. C. G. and Talay, D.},
  publisher = {Cambridge University Press},
  address   = {Cambridge},
  pages     = {181--214},
  year      = {1997},
  doi       = {10.1017/CBO9781139173056.010}
}

@article{DuffieEpstein1992a,
  author    = {Duffie, Darrell and Epstein, Larry G.},
  title     = {Stochastic differential utility},
  journal   = {Econometrica},
  volume    = {60},
  number    = {2},
  pages     = {353--394},
  year      = {1992}
}

@article{pengPDE,
title = {Probabilistic interpretation for systems of quasilinear parabolic partial differential equations},
author={Peng, Shige},
journal = {Stochastics and Stochastic Reports},
volume = {37},
number = {1-2},
pages = {61-74},
year  = {1991},
}


\end{document}